\documentclass{amsart}
\usepackage{amsmath,amsthm,thm-restate,thmtools,amsfonts,amssymb}
\usepackage{natbib}
\usepackage{hyperref,color,comment}

\newcommand{\out}{\mathsf{Out(\mathbb{F})}}
\newcommand{\aut}{\mathsf{Aut(\mathbb{F})}}
\newtheorem{thm}{Theorem}[section]
\newtheorem{prop}[thm]{Proposition}
\newtheorem{lemma}[thm]{Lemma}

\newtheorem{cor}[thm]{Corollary}

\newtheorem*{fact*}{Fact}

\theoremstyle{definition}
\newtheorem{definition}[thm]{Definition}

\newtheorem{remark}[thm]{Remark}

\title{Limits of conjugacy classes under iterates of Hyperbolic elements of $\out$}

\author{Pritam Ghosh}
\dedicatory{Dedicated to Lee Mosher on his $60^{th}$ birthday}
\begin{document}
 
 \maketitle

 \begin{abstract}
  For a free group $\mathbb{F}$ of finite rank such that $\text{rank}(\mathbb{F})\geq 3$,
 we prove that the set of weak limits of a conjugacy class in $\mathbb{F}$ under iterates of some
 hyperbolic $\phi\in\out$ is equal to the collection of generic leaves and  lines with endpoints in 
 attracting fixed points of $\phi$.

 As an application we describe the ending lamination set for a hyperbolic extension of $\mathbb{F}$ by a hyperbolic
 subgroup of $\out$ in a new way and use it to prove results about Cannon-Thurston maps for such
 extensions. We also use it to derive conditions for quasiconvexity of finitely generated, infinite index subgroups
 of $\mathbb{F}$ in the extension group. These results generalize similar results obtained in
 \cite{Mj-99, KL-13} and use different techniques.
\end{abstract}


 \section{Introduction}

 Fix a nonabelian free group $\mathbb{F}$ of finite rank and assume that $\text{rank}(\mathbb{F})$ $\geq 3$.
  The outer automorphism group, $\out$, of a free group $\mathbb{F}$ is defined as $\aut$ modulo the group of inner automorphisms
 of $\mathbb{F}$.
 There are many tools in studying the properties of this group. One of
them is by using train-track maps introduced in \cite{BH-92} and later generalized in \cite{BFH-97}, \cite{BFH-00}, \cite{FH-11}. The fully irreducible
outer automorphisms are elements in $\out$ which do not have iterates that leave a conjugacy class of some proper
free factor invariant. They are the most well understood elements in $\out$.
They behave very closely to the pseudo-Anosov homeomorphisms of surfaces with one boundary component,
which have been well understood and are a rich source of examples and interesting theorems. Our focus will be on outer automorphisms
$\phi$ which might not be fully irreducible but have a lift $\Phi\in\aut$ such that
$$\exists M>0, \lambda>1 \ni \lambda|c|\leq \text{ max }\{| \Phi^M(c) |, | \Phi^{-M}(c) |\} \forall c\in \mathbb{F}$$
Such classes of outer automorphisms are called \textit{hyperbolic} outer automorphisms.
It was shown by \cite{BF-92} and \cite{Br-00} that an outer automorphism is hyperbolic
if and only if it does not have any periodic conjugacy classes.

Our first objective (Section \ref{main} ) is to classify all \textit{weak} limits of any conjugacy class $[c]$ under the action of a hyperbolic
$\phi\in\out$. For this we use the theory of completely split relative train track maps introduced by
Feighn-Handel in their Recognition Theorem work \cite{FH-11} and later further studied in great details by
Handel-Mosher in their body of work about Subgroup Decomposition in $\out$ \cite{HM-13a, HM-13c}.
Our main result in this direction is the following result (here
 $\mathcal{B}_{gen}(\phi)$ denotes the set of all generic leaves of all attracting laminations for $\phi$ (See section \ref{sec:6}) and
$\mathcal{B}_{\mathsf{Fix}_+}(\phi)$  denotes the set of all lines with endpoints in $\mathsf{Fix}_+(\phi)$.(See section \ref{sec:8})):

 \newtheorem*{Limit}{Theorem \ref{Limit}}
\begin{Limit}

\textit{Suppose $\phi \in \out$ is a hyperbolic outer automorphism and $[c]$ is
 any conjugacy class in $\mathbb{F}$. Then the weak limits of $[c]$ under iterates of $\phi$ are in
 $$\mathcal{WL}(\phi) : =  \mathcal{B}_{gen}(\phi) \cup \mathcal{B}_{\mathsf{Fix}_+}(\phi)$$}
  \textit{Conversely, any line in $\mathcal{WL}(\phi)$ is a weak limit of some conjugacy class in $\mathbb{F}$ under iterates of $\phi$.}

 \end{Limit}

 The dynamics of fully irreducible outer automorphisms are very well understood.
 This theorem, in a way shows that
dynamical behavior of hyperbolic outer automorphisms which are not fully irreducible are also well behaved.
 One of the special properties of a fully irreducible and hyperbolic automorphism $\phi$ is that
 the non-attracting subgroup system (see section \ref{sec:7}) of the unique attracting lamination of $\phi$
 is trivial. This implies
 that every conjugacy class is attracted to the unique attracting lamination under iterates of $\phi$.
 The above theorem goes on to show that a somewhat similar behaviour is observed in the case of hyperbolic
 automorphisms which are not fully irreducible. More precisely, every conjugacy class is attracted to \textit{some}
 attracting lamination under iterates of $\phi$, which is the content of Lemma \ref{EG}.
 The above theorem also implies that $\phi$ is hyperbolic if and only if every conjugacy class is weakly
 attracted to some singular line of $\phi$ under iterates of $\phi$ (See \ref{althyp}).

 Theorem \ref{Limit} is used to prove several important results in Section \ref{appl} . The
 work of \cite[Theorem 1.2]{Br-00} and  \cite{BF-92} show that
 $\phi$ is hyperbolic if and only if the mapping torus $G_{\phi}: = \mathbb{F}_{\langle \phi \rangle} \rtimes \mathbb{Z}$
 is a Gromov hyperbolic group. We use the above theorem \ref{Limit} to define a set of lines which describe
 the ending lamination set ($\Lambda_{\langle \phi \rangle}$) defined  in \cite[]{Mj-97} for  $\mathbb{F}_{\langle \phi \rangle} \rtimes \mathbb{Z}$.
 We prove:

 \newtheorem*{el}{Lemma \ref{el1}}
 \begin{el}
  \textit{Let $\mathcal{H} = \langle \phi \rangle$ for some hyperbolic $\phi\in\out$. Then
 $$\Lambda_{\langle \phi \rangle} = \widetilde{\mathcal{WL}}(\phi) \cup \widetilde{\mathcal{WL}}(\phi^{-1})$$}

 \end{el}

 This new description of the ending lamination set has very significant implications. First we use this
 new description of the ending lamination set to prove the following theorem, which partially answers a question of Swarup:

 \newtheorem*{CThyp}{Theorem \ref{ct1}}
 \begin{CThyp}

  \textit{Consider the exact sequence of hyperbolic groups:}
$$ 1\rightarrow \mathbb{F} \rightarrow G_{\phi}\rightarrow \langle\phi\rangle\rightarrow 1$$

  \textit{The Cannon-Thurston map $\widehat{i}:\partial\mathbb{F}\
 \rightarrow \partial G_{\phi}$ is a finite-to-one map and cardinality of preimage of each point in $\partial G_{\phi}$ is
 uniformly bounded, the bound depending only on rank of $\mathbb{F}$.}

 \end{CThyp}

 The proof of this theorem becomes fairly simple once we prove in Proposition \ref{nongen0} that every
 nongeneric generic leaf of an attracting lamination for $\phi$ is a line in $\mathcal{B}_{\mathsf{Fix}_+}(\phi)$.
 Then we can apply the proposition that the number of lines in $\mathcal{B}_{\mathsf{Fix}_+}(\phi)$ for any hyperbolic outer automorphism $\phi$ is uniformly bounded
 above by some number depending only on rank of $\mathbb{F}$ (see Proposition \ref{singfin}).

 The special case, of the aforementioned theorem, when $\phi$ is fully irreducible was first proved by
  \cite[Theorem 5.4]{KL-13}. However the methods of our proofs are very different; we use the action of
  $\phi$ on conjugacy classes of $\mathbb{F}$ and their proof uses action of $\phi$ on boundary of outer space.

 In Section \ref{sec:A2} the next application of the new description of the ending lamination set is in understanding quasiconvexity of
 finitely generated infinite index subgroups of $\mathbb{F}$ inside $G_{\phi}$. Here we are able to
 give one equivalent condition and also one sufficient condition for quasiconvexity. (For definition of
 $\mathsf{Fix}_+(\phi)$ please see section \ref{sec:4}).

 \newtheorem*{qcl}{Theorem \ref{qc1}}
 \begin{qcl}
  \textit{Consider the exact sequence} $$ 1\rightarrow \mathbb{F} \rightarrow G_{\phi}\rightarrow \langle\phi\rangle\rightarrow 1$$
 \textit{where $\phi$ is a hyperbolic automorphism of $\mathbb{F}$. Let $H$ be a finitely generated infinite index subgroup of
 $\mathbb{F}$.  Then the following are equivalent: }
 \begin{enumerate}
  \item \textit{$H$ is quasiconvex in $G_{\phi}$. }
  \item $\partial H\cap \{\mathsf{Fix}_+(\phi)\cup \mathsf{Fix}_-(\phi)\} = \emptyset$.

 \end{enumerate}
 \textit{Moreover both $(1)$ and $(2)$ are satisfied if every attracting and repelling lamination of $\phi$ is
 minimally filling with respect to $H$.}

 \end{qcl}

 The proof of this theorem uses a result that was proved  in \cite{HM-09}
 \ref{structure}, which states that every attracting lamination $\Lambda$ of $\phi$  contains a leaf
 which is a singular line for $\phi$ and one of whose ends is dense in $\Lambda$.

 The concept of \textit{minimally filling with respect to $H$} is motivated from the idea of
 \textit{minimally filling}. We make this definition in \ref{relminfill} and show that
 this theorem indeed covers the case for a fully irreducible $\phi$ in Corollary \ref{qciwip}.

 Dynamics of outer automorphisms have been studied for quite sometime using the notion of train track maps which were introduced
 in the seminal paper of \cite{BH-92}. Then \cite[Proposition 2.4]{BFH-97} studied the dynamics of
 fully irreducible outer automorphisms and proved the existence of an attracting lamination with the property
 that any finitely generated subgroup of $\mathbb{F}$ which carries this lamination must be of finite index in $\mathbb{F}$.
 Note that this is a much stronger property than just ``filling'' the free group (not being carried by a proper free factor of $\mathbb{F}$)
 and was called \textit{minimally filling} by Kapovich-Lustig in their work \cite{KL-13}. We establish a connection between
 minimally filling and filling (in the sense of free factor support) in Proposition \ref{fill-minfill}.
 The primary obstruction that occurs is due to possible existence filling attracting laminations which
 are not minimally filling. We then proceed to generalize
 the notion of minimally filling for the purposes of using it for hyperbolic outer automorphisms which are not
 fully irreducible (Definition \ref{relminfill}).

 Next we give an algebraic condition on a finitely generated, infinite index subgroup of
 $\mathbb{F}$ for it to be quasiconvex in $G_{\phi}$ (see Theorem \ref{qc2}). It is quite understandable that such
 a general condition is quite complicated in nature due to complexities that arise in a non-fully irreducible hyperbolic
 $\phi$. For example, one might have a filling lamination which may not be minimally filling and
 whose nonattracting subgroup system is nontrivial implying that not every conjugacy class is attracted to it under
 iteration of $\phi$.

 Section \ref{sec:A3} of this paper deals with the Cannon-Thurston maps for a
 more complicated scenario when we have an exact sequence of hyperbolic groups
 $$ 1\rightarrow \mathbb{F} \rightarrow G \rightarrow \mathcal{H} \rightarrow 1$$
 where $\mathcal{H}$ is non-elementary hyperbolic subgroup of $\out$.  Observe that hyperbolicity of
 $G$ implies that $\mathcal{H}$ is \textit{purely atoroidal}, meaning every element of $\mathcal{H}$ is a hyperbolic outer automorphism.
 We use the Lemma \ref{el1}, where we describe the nature of ending lamination set for a hyperbolic $\phi$ in terms
 of generic leaves and lines in $\mathcal{B}_{\mathsf{Fix}_+}(\phi)$, to derive a list equivalent necessary conditions in Proposition
 \ref{nec} that elements of $\mathcal{H}$ must satisfy in order for the extension group to be hyperbolic. In particular
 it we observe that two hyperbolic outer automorphisms which do not have a common power but have a common attracting lamination
 cannot both belong to $\mathcal{H}$.

 In a very recent work \cite{DKT-16} have
 shown the finiteness of the fibers of Cannon-Thurston map for the special case when $\mathcal{H}$ is convex cocompact. The
 general case however still remains open. We hope that the results we develop here will be useful for addressing this problem
 in its full generality and also other interesting questions related to hyperbolicity in $\out$.

 \vspace{1cm}

 \textbf{Acknowledgement: } The author would like to thank Mahan Mj for suggesting the applications to
 quasiconvexity after reading the first draft of the paper. We also thank Pranab Sardar for help with the literature on
 quasiconvexity and other helpful discussions.

 The author was supported by SERB N-PDF grant PDF/2015/000038 during this project.


\section{Preliminaries}
\label{sec:1}
In this section we give the reader a short review of the definitions and some important results in $\out$ which are relevant to the theorem that we are trying to prove here. All the results which are stated as lemmas, can be found in full details in
\cite{BFH-00}, \cite{FH-11}, \cite{HM-13a}, \cite{HM-13c}.

\subsection{Weak topology}
\label{sec:2}
Given any finite graph $G$, let $\widehat{\mathcal{B}}(G)$ denote the compact space of equivalence classes of circuits in $G$ and paths in $G$, whose endpoints (if any) are vertices of $G$. We give this space the \textit{weak topology}.
Namely, for each finite path $\gamma$ in $G$, we have one basis element $\widehat{N}(G,\gamma)$ which contains all paths and circuits in $\widehat{\mathcal{B}}(G)$ which have $\gamma$ as its subpath.
Let $\mathcal{B}(G)\subset \widehat{\mathcal{B}}(G)$ be the compact subspace of all lines in $G$ with the induced topology. One can give an equivalent description of $\mathcal{B}(G)$ following \cite{BFH-00}.
A line is completely determined, upto reversal of direction, by two distinct points in $\partial \mathbb{F}$, since there only one line that joins these two points. 
We can then induce the weak topology on the set of lines coming from the Cantor set $\partial \mathbb{F}$. More explicitly,
let $\widetilde{\mathcal{B}}=\{ \partial \mathbb{F} \times \partial \mathbb{F} - \vartriangle \}/(\mathbb{Z}_2)$, where $\vartriangle$ is the diagonal and $\mathbb{Z}_2$ acts by interchanging factors. We can put the weak topology on
$\widetilde{\mathcal{B}}$, induced by Cantor topology on $\partial \mathbb{F}$. The group $\mathbb{F}$ acts on $\widetilde{\mathcal{B}}$ with a compact but non-Hausdorff quotient space $\mathcal{B}=\widetilde{\mathcal{B}}/\mathbb{F}$.
The quotient topology is also called the \textit{weak topology}. Elements of $\mathcal{B}$ are called \textit{lines}. A lift of a line $\gamma \in \mathcal{B}$ is an element  $\widetilde{\gamma}\in \widetilde{\mathcal{B}}$ that
projects to $\gamma$ under the quotient map and the two elements of $\widetilde{\gamma}$ are called its endpoints.

One can naturally identify the two spaces $\mathcal{B}(G)$ and $\mathcal{B}$ by considering a homeomorphism between the two Cantor sets $\partial \mathbb{F}$ and set of ends of universal cover of $G$ , where $G$ is a marked graph.
$\out$ has a natural action on  $\mathcal{B}$. The action comes from the action of Aut($\mathbb{F}$) on $\partial \mathbb{F}$. Given any two marked graphs $G,G'$ and a homotopy equivalence $f:G\rightarrow G'$ between them, the induced map
$f_\#: \widehat{\mathcal{B}}(G)\rightarrow \widehat{\mathcal{B}}(G')$ is continuous and the restriction $f_\#:\mathcal{B}(G)\rightarrow \mathcal{B}(G')$ is a homeomorphism. With respect to the identification
$\mathcal{B}(G)\approx \mathcal{B}\approx \mathcal{B}(G')$, if $f$ preserves the marking then $f_{\#}:\mathcal{B}(G)\rightarrow \mathcal{B}(G')$ is equal to the identity map on $\mathcal{B}$. When $G=G'$, $f_{\#}$ agree with their
homeomorphism $\mathcal{B}\rightarrow \mathcal{B}$ induced by the outer automorphism associated to $f$.

Given a marked graph $G$, a \textit{ray} in $G$ is an one-sided infinite concatenation of edges $E_0E_1E_2......$. A \textit{ray} of $\mathbb{F}$ is an
element of the orbit set $\partial\mathbb{F}/\mathbb{F}$. There is connection between these two objects which can be explained as follows.
Two  rays in $G$ are asymptotic if they have equal subrays, and
this is an equivalence relation on the set of rays in $G$. The set of asymptotic equivalence
classes of rays $\rho$ in $G$ is in natural bijection with $\partial\mathbb{F}/\mathbb{F}$ where $\rho$ in $G$ corresponds to end
$\xi\in\partial\mathbb{F}/\mathbb{F}$ if there is a lift $\tilde{\rho} \subset G$ of $\rho$ and a lift
$\tilde{\xi}\in\partial\mathbb{F}$ of $\xi$, such that $\tilde{\rho}$ converges to $\tilde{\xi}$ in the
Gromov compactification of $\tilde{G}$. A ray $\rho$ is often said to be the realization of $\xi$ if the above
conditions are satisfied.

A line(path) $\gamma$ is said to be \textit{weakly attracted} to a line(path) $\beta$ under the action of $\phi\in\out$, if the $\phi^k(\gamma)$ converges to $\beta$ in the weak topology. This is same as saying, for any given finite subpath of $\beta$, $\phi^k(\gamma)$
contains that subpath for some value of $k$; similarly if we have a homotopy equivalence $f:G\rightarrow G$,  a line(path) $\gamma$ is said to be \textit{weakly attracted} to a line(path) $\beta$ under the action of $f_{\#}$ if the $f_{\#}^k(\gamma)$ converges to $\beta$
in the weak topology. The \textit{accumulation set} of a ray $\gamma$ in $G$ is the set of lines  $l\in \mathcal{B}(G)$ which are elements of the weak closure of $\gamma$; which is same as saying every finite subpath of $l$
occurs infinitely many times as a subpath $\gamma$. The weak accumulation set of some $\xi\in\partial\mathbb{F}$ is the set of lines in the weak closure
of any of the asymptotic rays in its equivalence class. We call this the \textit{weak closure} of $\xi$.

A line $l$ is said to be \textit{birecurrent} if $l$ if it is in the weak closure of some positive subray of itself  and
it is in the weak closure of some negative subray of itself. This implies that any finite subpath of $l$ occurs infinitely often in either
direction.

\subsection{Free factor systems and subgroup systems}
\label{sec:3}
Given a finite collection $\{K_1, K_2,.....,K_s\}$ of subgroups of $\mathbb{F}$ , we say that this collection determines a \textit{free factorization} of $\mathbb{F}$ if $\mathbb{F}$ is the free product of these subgroups, that is,
$\mathbb{F} = K_1 * K_2 * .....* K_s$. The conjugacy class of a subgroup is denoted by [$K_i$].

A \textit{free factor system} is a finite collection of conjugacy classes of subgroups of $\mathbb{F}$ , $\mathcal{K}:=\{[K_1], [K_2],.... [K_p]\}$ such that there is a free factorization of $\mathbb{F}$ of the form
$\mathbb{F} = K_1 * K_2 * ....* B$, where $B$ is some finite rank subgroup of $\mathbb{F}$ (it may be trivial).

There is an action of $\out$ on the set of all conjugacy classes of subgroups of $\mathbb{F}$. This action induces an action of $\out$ on the set of all free factor systems. For notation simplicity we will avoid writing $[K]$ all the time and write $K$ instead, when we discuss
the action of $\out$ on this conjugacy class of subgroup $K$ or anything regarding the conjugacy class [$K$]. It will be understood that we actually mean [$K$].

For any marked graph $G$ and any subgraph $H \subset G$, the fundamental groups of the noncontractible components of $H$ form a free factor system . We denote this by $[H]$. A subgraph of $G$ which has no valence 1 vertex is called a \textit{core graph}.
Every subgraph has a unique core graph, which is a deformation retract of the union of the noncontractible components of $H$, implying that free factor system defined
by the core of $H$ is equal to the free factor system defined by core of $H$. Conversely, any free factor system can be realized as
$[H]$ for some nontrivial core subgraph of some marked graph $G$.

A free factor system $\mathcal{K}$ \textit{carries a conjugacy class} $[c]$ in $\mathbb{F}$ if there exists some $[K] \in \mathcal{K}$ such that $c\in K$. We say that $\mathcal{K}$ \textit{carries the line} $\gamma \in \mathcal{B}$ if for any marked graph $G$
the realization of $\gamma$ in $G$ is the weak limit of a sequence of circuits in $G$ each of which is carried by $\mathcal{K}$.
An equivalent way of saying this is: for any marked graph $G$ and a subgraph $H \subset G $ with $[H]=\mathcal{K}$, the realization of $\gamma$ in $G$ is contained in $H$.

Similarly define a \textit{subgrpoup system} $\mathcal{A} = \{[H_1], [H_2], .... ,[H_k]\}$ to be a finite collection of conjugacy classes of finite rank subgroups $H_i<\mathbb{F}$.

A subgroup system $\mathcal{A}$ carries a conjugacy class $[c]\in \mathbb{F}$ if there exists some $[A]\in\mathcal{A}$ such that $c\in A$. Also, we say that $\mathcal{A}$ carries a line $\gamma$ if one of the following equivalent conditions hold:
\begin{itemize}
\item $\gamma$ is the weak limit of a sequence of conjugacy classes carries by $\mathcal{A}$.
\item There exists some $[A]\in \mathcal{A}$ and a lift $\widetilde{\gamma}$ of $\gamma$ so that the endpoints of $\widetilde{\gamma}$ are in $\partial A$.

\end{itemize}
The following lemma is an important property of lines carried by a subgroup system. The proof is by using the observation that $A<\mathbb{F}$ is of finite rank implies that $\partial A$ is a compact subset of $\partial \mathbb{F}$.
\begin{lemma}
For each subgroup system $\mathcal{A}$ the set of lines carried by $\mathcal{A}$ is a closed subset of $\mathcal{B}$
\end{lemma}


From \cite{BFH-00} \label{fill}
The \textit{free factor support} of a set of lines $B$ in $\mathcal{B}$ is (denoted by $\mathcal{A}_{supp}(B)$) defined as the \textit{meet} of all free factor systems that carries $B$. We are skipping giving the exact definition of \textit{meet} here
since we have no explicit use for that definition. Roughly speaking, one should think of the free factor support as the smallest (in terms of inclusion of subgroups)
free factor system that carries $B$ (for more details please see \cite[Corollary 2.6.5]{BFH-00}).
If $B$ is a single line then $\mathcal{A}_{supp}(B)$ is single free factor. 
We say that a set of lines, $B$, is \textit{filling} if $\mathcal{A}_{supp}(B)=[\mathbb{F}]$.

\subsection{Principal automorphisms and rotationless automorphisms}
\label{sec:4}

Given an outer automorphism $\phi\in\out$ , we can consider a lift $\Phi$ in $\aut$. We call a lift principal automorphism, if it satisfies certain conditions described below. Roughly speaking, what such lifts guarantees is the the existence of certain lines which are not a part of the attracting lamination but it still
fills the free group $\mathbb{F}$. Such lines (called \textit{singular lines}) will be a key tool in describing the set of lines which are not attracted to the attracting lamination of $\phi$.\\
Consider $\phi\in \out$ and a lift $\Phi$ in $\aut$. $\Phi$ has an action on $\mathbb{F}$, which has a fixed subgroup denoted by Fix$(\Phi)$. Consider the boundary of this fixed subgroup $\partial$Fix$(\Phi)\subset$ $\partial \mathbb{F}$. It is either empty or has exactly two points.\\
This action action extends to the boundary and is denoted by $\widehat{\Phi}:\partial \mathbb{F}\rightarrow\partial \mathbb{F}$. Let Fix$(\widehat{\Phi})$ denote the set of fixed points of this action. We call an element $P$ of Fix$(\widehat{\Phi})$ \textit{attracting fixed point} if there exists an open neighborhood $U\subset \partial \mathbb{F}$ of $P$ such that we have $\widehat{\Phi}(U)\subset U$
and for every points $Q\in U$ the sequence $\widehat{\Phi}^n(Q)$ converges to $P$. Let Fix$_+(\widehat{\Phi})$ denote the set of attracting fixed points of Fix$(\widehat{\Phi})$. Similarly let Fix$_-(\widehat{\Phi})$ denote the attracting fixed points of Fix$(\widehat{\Phi}^{-1})$. \\
Let Fix$_N(\widehat{\Phi})=$ Fix$(\widehat{\Phi})$ $-$ Fix$_-(\widehat{\Phi})= \partial$Fix$(\Phi)$ $\cup$ Fix$_+(\widehat{\Phi})$. We say that an automorphism $\Phi\in\aut$ in the outer automorphism class of $\phi$ is a \textit{principal automorphism} if
Fix$_N(\widehat{\Phi})$ has at least 3 points or Fix$_N(\widehat{\Phi})$ has exactly two points which are neither the endpoints of an axis of a covering translation, nor the endpoints of a
 generic leaf of an attracting lamination $\Lambda^+_\phi$. The set of all principal automorphisms of $\phi$ is denoted by $P(\phi)$. \\

 \textbf{Notation:} For sake of simplicity we will abuse the notation in this paper. Since we will be dealing exclusively with
 hyperbolic outer automorphisms in this paper we know that $\partial$Fix$(\Phi)$ is trivial for any automorphism $\Phi$ in the class of
 $\phi$. Hence Fix$_N(\widehat{\Phi})=$ Fix$_+(\widehat{\Phi})$. We shall abbreviate
 $$\mathsf{Fix}_+(\phi): = \bigcup_{\Phi\in P(\phi)} \mathsf{Fix}_+(\widehat{\Phi})$$
  Similarly for $\mathsf{Fix}_-(\phi)$.

We have the following lemma from \cite{LL-08} that we shall be using in this paper to show that any conjugacy class limits to either a
generic leaf of some attracting lamination of $\phi$ or a singular line of $\phi$ (see proposition \ref{nongen0}). This result essentially tells us that
the attracting fixed points are almost globally attracting, except for the finite number of repelling fixed points.

\begin{lemma}\label{attfix}

 \cite[Theorem I]{LL-08}
 If $\phi\in\out$ has no periodic conjugacy classes in $\mathbb{F}$ then there exists an integer $q\geq 1$ such that
 for each $\Phi\in\aut$ representing $\phi$ and each $\xi\in\partial \mathbb{F}$, one of the following holds:
 \begin{enumerate}
  \item $\xi\in$ $\mathsf{Fix}_-(\widehat{\Phi}^q)$.
  \item The sequence $\widehat{\Phi}^{qi}$ converges to a point in $\mathsf{Fix}_+(\widehat{\Phi}^q)$.
 \end{enumerate}

\end{lemma}

Let Per$(\widehat{\Phi})$ = $\cup_{k\geq1}$Fix$(\widehat{\Phi}^k)$, Per$_+(\widehat{\Phi})$ = $\cup_{k\geq1}$Fix$_+(\widehat{\Phi}^k)$ and similarly define Per$_-(\widehat{\Phi})$ and Per$_N(\widehat{\Phi})$. \\
We say that $\phi\in \out$ is rotationless if Fix$_N(\widehat{\Phi})$ = Per$_N(\widehat{\Phi})$ for all $\Phi\in P(\phi)$, and if for each $k\geq1$ the map $\Phi\rightarrow \Phi^k$ induces a bijection between $P(\phi)$ and $P(\phi^k)$. \\
The following two important lemmas about rotationless automorphisms are taken from \cite{FH-11}.

\begin{lemma}[\cite{FH-11},Lemma 4.43]
\label{rotationless}
 There exists a $K$ depending only upon the rank of the free group $\mathbb{F}$ such that for every $\phi\in \out$ , $\phi^K$ is rotationless.
\end{lemma}

The above lemma is heavily used in this paper. Whenever we write \textquotedblleft pass to a rotationless power \textquotedblright we intend to use this uniform constant $K$ given by the lemma.
Rotationless powers
are useful and important since they kill any periodic behaviour (as the following lemma  shows) and guarantee the existence of completely split relative train track maps as we shall see later.
\begin{lemma}[\cite{FH-11}]
 If $\phi\in \out$ is rotationless then:
\begin{itemize}
 \item Every periodic conjugacy class of $\phi$ is a fixed conjugacy class.
\item Every free factor system which is periodic under $\phi$ is fixed.
\item Every periodic direction of $\phi$ is fixed by $\phi$.
\end{itemize}

\end{lemma}

\begin{remark}

  These results show that given any $\phi\in\out$ one can pass to $\psi = \phi^K$ which is rotationless and hence one can
  choose $q=1$ for the rotationless outer automorphism $\psi$ as in Lemma \ref{attfix}. This is the form in which we will be using
  lemma \ref{attfix}.
 \end{remark}

\subsection{Topological representatives and Train track maps}
\label{sec:5}
Given $\phi\in\out$ a \textit{topological representative} is a homotopy equivalence $f:G\rightarrow G$ such that $\rho: R_r \rightarrow G$ is a marked graph, 
$f$ takes vertices to vertices and edges to paths and $\overline{\rho}\circ f \circ \rho: R_r \rightarrow R_r$ represents $R_r$. A nontrivial path $\gamma$ in $G$ 
is a \textit{periodic Nielsen path} if there exists a $k$ such that $f^k_\#(\gamma)=\gamma$;
the minimal such $k$ is called the period and if $k=1$, we call such a path \textit{Nielsen path}. A periodic Nielsen path is \textit{indivisible} if it cannot 
be written as a concatenation of two or more nontrivial periodic Nielsen paths.

Given a subgraph $H\subset G$ let $G\setminus H$ denote the union of edges in $G$ that are not in $H$.

Given a marked graph $G$ and a homotopy equivalence $f:G\rightarrow G$ that takes edges to paths, one can define a new map $Tf$ by setting $Tf(E)$ 
to be the first edge in the edge path associated to $f(E)$; similarly let $Tf(E_i,E_j) = (Tf(E_i),Tf(E_j))$. So $Tf$ is a map that takes turns to turns. We say that a 
nondegenerate turn is illegal if for some iterate of $Tf$ the turn becomes degenerate; otherwise the
 turn is legal. A path is said to be legal if it contains only legal turns and it is $r-legal$ if it is of height $r$ and all its illegal turns are in $G_{r-1}$.

 \textbf{Relative train track map.} Given $\phi\in \out$ and a topological representative $f:G\rightarrow G$ with a filtration $G_0\subset G_1\subset \cdot\cdot\cdot\subset G_k$ which is preserved by $f$, we say that $f$ is a train relative train track map if the following conditions are satisfied:
 \begin{enumerate}
  \item $f$ maps r-legal paths to legal r-paths.
  \item If $\gamma$ is a path in $G$ with its endpoints in $H_r$ then $f_\#(\gamma)$ has its end points in $H_r$.
  \item If $E$ is an edge in $H_r$ then $Tf(E)$ is an edge in $H_r$
 \end{enumerate}

 For any topological representative $f:G\rightarrow G$ and exponentially growing stratum $H_r$, let $N(f,r)$ be the number of indivisible 
 Nielsen paths $\rho\subset G$ that intersect the interior of $H_r$. Let $N(f)= \Sigma_r N(f,r)$. Let $N_{min}$ be the minimum value of $N(f)$ that occurs among the topological representatives with $\Gamma=\Gamma_{min}$. We call a relative train track map stable if $\Gamma=\Gamma_{min}$ and $N(f)=N_{min}$.
 The following result is Theorem 5.12 in \cite{BH-92} which assures the existence of a stable relative train track map.
 \begin{lemma}
  Every $\phi\in \out$ has a stable relative train track representative.
 \end{lemma}

 \textbf{Splittings, complete splittings and CT's}. Given relative train track map $f:G\rightarrow G$, splitting of a line, path or a circuit $\gamma$ is a decomposition of $\gamma$ into subpaths $....\gamma_0\gamma_1 .....\gamma_k....  $ 
 such that for all $i\geq 1$ the path $f^i_\#(\gamma) =  .. f^i_\#(\gamma_0)f^i_\#(\gamma_1)...f^i_\#(\gamma_k)...$
 The terms $\gamma_i$ are called the \textit{terms} of the splitting of $\gamma$.

 Given two linear edges $E_1,E_2$ and a root-free closed Nielsen path $\rho$ such that $f_\#(E_i) = E_i.\rho^{p_i}$ then we say that $E_1,E_2$ are said to be in the \textit{same linear family} and any path of the form $E_1\rho^m\overline{E}_2$ for some integer $m$ is called an \textit{exceptional path}.

 \textbf{Complete splittings:} A splitting of a path or circuit $\gamma = \gamma_1\cdot\gamma_2......\cdot \gamma_k$ is called complete splitting if each term $\gamma_i$ falls into one of the following categories:
 \begin{itemize}
  \item $\gamma_i$ is an edge in some irreducible stratum.
  \item $\gamma_i$ is an indivisible Nielsen path.
  \item $\gamma_i$ is an exceptional path.
  \item $\gamma_i$ is a maximal subpath of $\gamma$ in a zero stratum $H_r$ and $\gamma_i$ is taken.
 \end{itemize}

 \textbf{Completely split improved relative train track maps}. A \textit{CT} or a completely split improved relative train track maps are topological representatives with particularly nice properties. But CTs do not exist for all outer automorphisms. Only the rotationless outer automorphisms are guaranteed to have a CT representative
 as has been shown in the following Theorem from \cite{FH-11}(Theorem 4.28).
 \begin{lemma}
  For each rotationless $\phi\in \out$ and each increasing sequence $\mathcal{F}$ of $\phi$-invariant free factor systems, there exists a CT $f:G\rightarrow G$ that is a topological
  representative for $\phi$ and $f$ realizes $\mathcal{F}$.
 \end{lemma}

 The following results are some properties of CT's defined in Recognition theorem work of Feighn-Handel in \cite{FH-11}.
 We will state only the ones we need here.
 \begin{enumerate}\label{CT}
  \item \textbf{(Rotationless)} Each principal vertex is fixed by $f$ and each periodic direction at a principal vertex is fixed by $Tf$.
  \item \textbf{(Completely Split)} For each edge $E$ in each irreducible stratum, the path $f(E)$ is completely split.
  \item \textbf{(vertices)} The endpoints of all indivisible Nielsen paths are vertices. The terminal endpoint of each nonfixed NEG edge is principal.
  \item \textbf{(Periodic edges)} Each periodic edge is fixed.
  \item \textbf{(Zero strata)} Each zero strata $H_i$ is contractible and enveloped by a EG strata
  $H_s, s>i$, such that every edge of $H_i$ is a taken in $H_s$. Each vertex of $H_i$ is contained in $H_s$ and
  link of each vertex in $H_i$ is contained in $H_i \cup H_s$.
  \item \textbf{(Linear Edges)} For each linear edge $E_i$ there exists a root free indivisible Nielsen
  path $w_i$ such that $f_\#(E_i) = E_i w^{d_i}_i$ for some $d_i \neq 0$.
  \item \textbf{(Nonlinear NEG edges)} \cite[Lemma 4.21]{FH-11} Each non-fixed NEG stratum $H_i$ is a single edge with its
  NEG orientation and has a splitting $f_\#(E_i) = E_i\cdotp u_i$, where $u_i$ is a closed nontrivial  completely split
  circuit and is an indivisible Nielsen path if and only if $H_i$ is linear.
 \end{enumerate}

CTs have very nice properties. The reader can look them up \cite{FH-11} for a detailed exposition or \cite{HM-13a} for a quick reference. We list below only a few of them that is needed for us.
All three lemmas are present in the aforementioned papers. These properties along with the complete description
of the components in a complete splitting of a circuit are the main reasons why will keep working with rotationless powers
of a hyperbolic $\phi$ in the next two sections. We can achieve a great deal of control when we iterate
random conjugacy classes under $\phi$.
\begin{lemma}(\cite{FH-11}, Lemma 4.11)
 A completely split path or circuit has a unique complete splitting.
\end{lemma}

The following lemma is a crucial component of our proof here. It assures us that as we iterate a circuit $\sigma$ under $f_\#$ eventually we
will achieve a complete splitting at some point and then analyzing the components of such a splitting (as we do in Lemma \ref{EG}) tells us what the possible limits
could be.
\begin{lemma}\label{circuitsplit}
\cite[Lemma 4.26]{FH-11}
 If $\sigma$ is a finite path or a circuit with endpoint in vertices, then $f^k_\#(\sigma)$ is completely split for all sufficiently large $k\geq 1$
 and $f^{k+1}_\#(\sigma)$ has complete splitting that is a refinement of the complete splitting of
 $f^k_\#(\sigma)$.
\end{lemma}

The following two Lemmas are used in proof of Lemma \ref{EG}, when we give our argument with
the zero strata within the proof. When a circuit $\sigma$ is iterated under a hyperbolic $\phi$ and we
have a complete splitting for some $f^k_\#(\sigma)$, where one of the components of the complete splitting is a path in a zero strata, these lemmas along with Property(5) listed above, tell us that
the preceding and the following components (with respect to the component contained in zero strata) in the complete splitting must be an exponentially growing edge.

\begin{lemma} \cite[Theorem 5.15, eg(i)]{BFH-00}\label{np}
 Every periodic Nielsen path is fixed. Also, for each EG stratum $H_r$ there exists at most one indivisible Nielsen path of height $r$, upto reversal of orientation, and
 the initial and terminal edges of this Nielsen path is in $H_r$.
\end{lemma}

The following lemma is part of \cite[Lemma 4.24]{FH-11}
\begin{lemma}\label{propzero}
  If an EG stratum $H_i$ has an indivisible Nielsen path of height $i$ then there is no zero strata enveloped by $H_i$
\end{lemma}

\subsection{Attracting Laminations and their properties under CTs}
\label{sec:6}
For any marked graph $G$, the natural identification $\mathcal{B}\approx \mathcal{B}(G)$ induces a bijection between the closed subsets of $\mathcal{B}$ and the closed subsets of $\mathcal{B}(G)$. A closed
subset in any of these two cases is called a \textit{lamination}, denoted by $\Lambda$. Given a lamination $\Lambda\subset \mathcal{B}$ we look at the corresponding lamination in $\mathcal{B}(G)$ as the
realization of $\Lambda$ in $G$. An element $\lambda\in \Lambda$ is called a \textit{leaf} of the lamination.\\
A lamination $\Lambda$ is called an \textit{attracting lamination} for $\phi$ is it is the weak closure of a line $l$ (called the \textit{generic leaf of $\lambda$}) satisfying the following conditions:
\begin{itemize}
 \item $l$ is birecurrent leaf of $\Lambda$.
\item $l$ has an \textit{attracting neighborhood} $V$, in the weak topology, with the property that every line in $V$ is weakly attracted to $l$.
\item no lift $\widetilde{l}\in \mathcal{B}$ of $l$ is the axis of a generator of a rank 1 free factor of $\mathbb{F}$ .
\end{itemize}

We know from \cite{BFH-00} that with each $\phi\in \out$ we have a finite set of laminations $\mathcal{L}(\phi)$, called the set of \textit{attracting laminations} of $\phi$, and the set $\mathcal{L}(\phi)$ is
invariant under the action of $\phi$. When it is nonempty $\phi$ can permute the elements of $\mathcal{L}(\phi)$ if $\phi$ is not rotationless. For rotationless $\phi$ $\mathcal{L}(\phi)$ is a fixed set. Attracting laminations are directly related to EG stratas.
An important result from \cite{BFH-00} section 3 is that there is a unique bijection between exponentially growing strata and
attracting laminations, which implies that there are only finitely many elements in $\mathcal{L}(\phi)$.

\textbf{Dual lamination pairs. }
We have already seen that the set of lines carried by a free factor system is a closed set and so, together with the lemma that the weak closure of a generic leaf $\lambda$ of an attracting lamination $\Lambda$ is the whole lamination $\Lambda$ tells us that
$\mathcal{A}_{supp}(\lambda) = \mathcal{A}_{supp}(\Lambda)$. In particular the free factor support of an attracting lamination $\Lambda$ is a single free factor.
Let $\phi\in \out$ be an outer automorphism and $\Lambda^+_\phi$ be an attracting lamination of $\phi$ and $\Lambda^-_\phi$ be an attracting lamination of $\phi^{-1}$. We say that this lamination pair is a \textit{dual lamination pair} if $\mathcal{A}_{supp}(\Lambda^+_\phi) = \mathcal{A}_{supp}(\Lambda^-_\phi)$.
By Lemma 3.2.4 of \cite{BFH-00} there is bijection between $\mathcal{L}(\phi)$ and $\mathcal{L}(\phi^{-1})$ induced by this duality relation.
\vspace{.2cm}

\textbf{Tiles: } Bestvina-Feighn-Handel introduced the concept of \textit{tiles} in \cite{BFH-00}.
For an edge $E$ in a EG stratum, an unoriented path of the form $f^k_\#(E)$ is called a \textit{k-tile of height r}.
We state below a two important applications of tiles.

\begin{lemma}
 \begin{enumerate}\label{tiles}
  \item \cite[Lemma 3.1.10 item (4)]{BFH-00} If $\Lambda_r$ is the unique attracting lamination associated with $H_r$ then every generic leaf can be written as
  a increasing union of tiles of height $r$.
  \item \cite[Lemma1.57 item(4)]{HM-13a} There exists $p$ such that for every $k\geq i\geq 0$, each
  $k+p$-tile of height $r$ contains every $i$-tile of height $r$.
 \end{enumerate}

\end{lemma}

This lemma will be used by us while analysing the weak limits of a conjugacy class under iterates of $\phi$; when we look at
a complete splitting of some $\phi^k_\#(\sigma)$ and one of the components in that splitting  is an edge in a EG strata, these lemma tell us that such a circuit is weakly attracted to the
attracting lamination related to that EG strata.

\subsection{Nonattracting subgroup system $\mathcal{A}_{na}(\Lambda^+_\phi)$}
\label{sec:7}
The \textit{nonattracting subgroup system} of an attracting lamination contains information about lines and circuits which are not attracted to the lamination.
The definition of this subgroup system is

\begin{definition}
 Suppose $\phi\in\out$ is rotationless and $f:G\rightarrow G$ is a CT representing $\phi$ such that $\Lambda^+_\phi$ is an invariant attracting lamination which corresponds to the EG stratum $H_s\in G$.
 The \textit{nonattracting subgraph Z} of $G$ is defined as a union of irreducible stratas $H_i$ of $G$ such that no edge in $H_i$ is weakly attracted to $\Lambda^+_\phi$. This is equivalent to saying that a strata $H_r\subset G\setminus Z$ if and only if there exists $k\geq 0$
 some term in the complete splitting of $f^k_\#(E_r)$ is an edge in $H_s$. Define the path $\widehat{\rho}_s$ to be trivial path at any chosen vertex if there does not exist any indivisible Nielsen path of height $s$, otherwise $\widehat{\rho}_s$ is the unique closed indivisible path of height $s$ (from definition of stable train track maps).
\end{definition}

\textbf{The groupoid  $\langle Z, \widehat{\rho}_s \rangle$ - } Let $\langle Z, \widehat{\rho}_s \rangle$ be the set of lines, rays, circuits and finite paths in $G$ which can be written as a concatenation of subpaths, each of which is an edge in $Z$, the path $\widehat{\rho}_s$ or its inverse. Under the operation of tightened concatenation of paths in $G$,
this set forms a groupoid (Lemma 5.6, [\cite{HM-13c}]).


Define the graph $K$ by setting $K=Z$ if $\widehat{\rho}_s$ is trivial and let $h:K \rightarrow G$ be the inclusion map. Otherwise define an edge $E_\rho$ representing the domain of the Nielsen path $\rho_s:E_\rho \rightarrow G_s$, and let $K$ be the disjoint union of $Z$ and $E_\rho$ with the following identification.
 Given an endpoint $x\in E_\rho$, if $\rho_s(x)\in Z$ then identify $x\sim\rho_s(x)$.Given distinct endpoints $x,y\in E_\rho$, if $\rho_s(x)=\rho_s(y)\notin Z$ then identify $x \sim y$. In this case define $h:K\rightarrow G$ to be the inclusion map on $K$ and the map $\rho_s$ on $E_\rho$. It is not difficult to see that the map $h$ is an immersion.
 Hence restricting $h$ to each component of $K$, we get an injection at the level of fundamental groups. The \textit{nonattracting subgroup system} $\mathcal{A}_{na}(\Lambda^+_\phi)$ is defined to be the subgroup system defined by this immersion.

We will leave it to the reader to look it up in \cite{HM-13c} where it is explored in details. We however list some key properties which we will be using and justifies the importance of this
subgroup system.
\begin{lemma}(\cite{HM-13c}- Lemma 1.5, 1.6)
\label{NAS}
 \begin{enumerate}
  \item The set of lines carried by $\mathcal{A}_{na}(\Lambda^+_\phi)$ is closed in the weak topology.
  \item A conjugacy class $[c]$ is not attracted to $\Lambda^+_\phi$ if and only if it is carried by $\mathcal{A}_{na}(\Lambda^+_\phi)$.
  \item $\mathcal{A}_{na}(\Lambda^+_\phi)$ does not depend on the choice of the CT representing $\phi$.
  \item  Given $\phi, \phi^{-1} \in \out$ both rotationless elements and a dual lamination pair $\Lambda^\pm_\phi$ we have $\mathcal{A}_{na}(\Lambda^+_\phi)= \mathcal{A}_{na}(\Lambda^-_\phi)$
  \item $\mathcal{A}_{na}(\Lambda^+_\phi)$ is a free factor system if and only if the stratum $H_r$ is not geometric.
  \item $\mathcal{A}_{na}(\Lambda^+_\phi)$ is malnormal.
  \item If $\{\gamma_n\}_{n\in\mathbb{N}}$ is a sequence of lines such that every weak limit of every subsequence of $\{\gamma_n\}$ is carried by $\mathcal{A}_{na}(\Lambda_\phi)$ then $\{\gamma_n\}$ is carried by $\mathcal{A}_{na}(\Lambda_\phi)$ for all sufficiently large $n$

 \end{enumerate}

\end{lemma}

\subsection{Singular lines, Extended boundary and Weak attraction theorem}
\label{sec:8}
In this section we will look at some results from \cite{HM-13c} which analyze and identify the set of lines which are not weakly attracted to an attracting lamination $\Lambda^\pm_\phi$, given some exponentially growing element in $\out$. Most of the results stated here are in terms of rotationless elements as in the original work. However, we note that being weakly attracted
to a lamination $\Lambda_\phi$ is not dependent on whether the element is rotationless. All lemmas stated here about rotationless elements also hold for non rotationless elements also, unless otherwise mentioned. This has been pointed out in Remark 5.1 in \cite{HM-13c}
The main reason for using rotationless elements is to make use of the train track structure from the CT theory. We will use some of the lemmas to prove lemmas about non rotationless elements which we will need later on.

Denote the set of lines not attracted to $\Lambda^+_\phi$ by $\mathcal{B}_{na}(\Lambda^+_\phi)$. The non-attracting subgroup system carries partial information about such lines as we can see in Lemma \ref{NAS}. Other obvious lines which are not attracted are the generic leaves of $\Lambda^-_\phi$. 
There is another class of lines, called singular lines, which we define below, which are not weakly attracted to $\Lambda^+_\phi$.

\begin{definition}
Define a \textit{\textbf{singular line}} for $\phi$ to be a line $\gamma\in \mathcal{B}$ if there exists a principal lift $\Phi$ of some rotationless iterate of $\phi$
 and a lift $\widetilde{\gamma}$ of $\gamma$ such that the endpoints of $\widetilde{\gamma}$ are contained in Fix$_+(\widehat{\Phi}) \subset \partial \mathbb{F}$.\\
The set of all singular lines of $\phi$ is denoted by $\mathcal{B}_{sing}(\phi)$.

\end{definition}

The lemma [Lemma 2.1, \cite{HM-13c}] below summarizes this discussion.
\begin{lemma}
 Given a rotationless $\phi\in \out$ and an attracting lamination $\Lambda^+_\phi$, any line $\gamma$ that satisfies one of the following three conditions is in $\mathcal{B}_{na}(\Lambda^+_\phi)$.
\begin{enumerate}
 \item $\gamma$ is carried by $\mathcal{A}_{na}\Lambda^\pm_\phi$
  \item $\gamma$ is a generic leaf of some attracting lamination for $\phi^{-1}$
\item $\gamma$ is in $\mathcal{B}_{sing}(\phi^{-1})$.
\end{enumerate}

\end{lemma}

But these are not all lines that constitute $\mathcal{B}_{na}(\Lambda^+_\phi)$. A important theorem in [Theorem 2.6, \cite{HM-13c}, stated below, tells us that there is way to concatenate lines from the three classes we mentioned in the above lemma which will also result in lines that are not weakly attracted to $\Lambda^+_\phi$. These are all possible types of lines in
$\mathcal{B}_{na}(\Lambda^+_\phi)$. A simple explanation of why the concatenation is necessary is, one can construct a line by connecting the base points of two rays, one of which is asymptotic to a singular ray in the forward direction of $\phi$ and the other is asymptotic to a singular ray in the backward direction of $\phi$. 
This line does not fall into any of the three categories we see in the lemma above.
The concatenation process described in \cite{HM-13c} takes care of such lines. We will not describe the concatenation here, but the reader can look up section 2.2 in \cite{HM-13c}. The following definition is by Handel and Mosher:

\begin{definition}

Let $A \in \mathcal{A}_{na}\Lambda^\pm_\phi$ and $\Phi \in P(\phi)$, we say that $\Phi$ is $A-related$ if Fix$_N(\widehat{\Phi})\cap \partial A \neq \emptyset$. Define the extended boundary of $A$ to be $$\partial_{ext}(A,\phi) = \partial A \cup \bigg( \bigcup_{\Phi}Fix_N(\widehat{\Phi}) \bigg)$$
where the union is taken over all $A$-related $\Phi\in P(\phi)$.
\end{definition}
Let $\mathcal{B}_{ext}(A,\phi)$ denote the set of lines which have end points in $\partial_{ext}(A,\phi)$; this set is independent of the choice of $A$ in its conjugacy class. Define $$\mathcal{B}_{ext}(\Lambda^+_\phi)  = \bigcup_{A\in \mathcal{A}_{na}\Lambda^\pm_\phi} \mathcal{B}_{ext}(A,\phi)$$

It is worth noting that the sets of lines mentioned in Lemma \ref{concat} are not necessarily pairwise disjoint. But if we have a line $\sigma \in \mathcal{B}_{na}(\Lambda^+_\phi)$ that is birecurrent then the situation is much simpler.
In that case $\sigma$ is either a generic leaf of some attracting lamination for $\phi^{-1}$ or $\sigma$ is carried by $\mathcal{A}_{na}\Lambda^\pm_\phi$. This takes us to the
following result due to Handel and Mosher that we need to prove an important result about asymptotic behaviour of leaves of attracting laminations
in Proposition \ref{asym2}, where we use the conclusion from the ``moreover'' part of the theorem to conclude that if
two leaves of attracting laminations of $\phi$ can be asymptotic then they are both singular lines of  $\phi$.

\begin{lemma}\label{concat}
 \cite[Theorem G]{HM-13c}

 If $\phi, \psi=\psi^{-1} \in \out$ are rotationless and $\Lambda^\pm_\phi$ is a dual lamination pair then
 $$\mathcal{B}_{na}(\Lambda^-_\phi, \psi) = \mathcal{B}_{ext}(\Lambda^\pm_\phi,\phi) \cup \mathcal{B}_{gen}(\phi) \cup \mathcal{B}_{sing}(\phi)$$

 Moreover the set of lines in $\mathcal{B}_{na}(\Lambda^-_\phi, \psi)$ are closed under concatenation. More precisely
 if $l', l''$ are two lines in $\mathcal{B}_{na}(\Lambda^-_\phi, \psi)$ with one asymptotic end $P$ say, and two other distinct
 endpoints $Q', Q''$ for $l', l''$ respectively,
 then either there exists some $[A]\in\mathcal{A}_{na}(\Lambda^\pm_\phi)$ so that $P, Q', Q'' \in \partial_{ext}(A,\phi)$ or there is a
 principal lift $\Phi$ such that all three points $P, Q', Q''$ are in $\mathsf{Fix}_N(\widehat{\Phi})$.

\end{lemma}
This result also plays a key role in the proof of Proposition \ref{Limit} where we show that weak limit of any conjugacy class under
iterates of a hyperbolic outer automorphism is either a generic leaf or a singular line and uses the understanding
of $\mathcal{B}_{ext}(\Lambda^\pm_\phi,\phi)$.
\begin{remark}
 For our purposes, where $\phi$ is hyperbolic, $\mathsf{Fix}_N(\widehat{\Phi}) = \mathsf{Fix}_+(\widehat{\Phi})$.
\end{remark}

From the work of Feighn and Handel in Recognition theorem we extract  the following lemma by assuming
$\phi$ is hyperbolic. The original statement is much more general.
 \begin{lemma}\label{lemma4}
 \cite[Lemma 4.36]{FH-11}
  Let CT $f: G\longrightarrow G$ represent a rotationless $\phi\in \out$, and $\tilde{f}: \tilde{G}\longrightarrow \tilde{G}$
  is a lift. If a vertex $v$ is fixed by $f$ and $E$ is a non-fixed edge in some EG strata or superlinear NEG strata 
  originating at $v$ with a fixed initial direction,  such that
  $f_\#(E) = Eu$, then there exists a splitting $f_\#(E) = E\cdotp u$ such that:
  $\widetilde{R} = \widetilde{E}\cdotp \tilde{u}\cdotp \tilde{f}_\#(\tilde{u})\cdotp \tilde{f}_\#^2(\tilde{u})...... \cdotp \tilde{f}^k_\#(\tilde{u})....$
  is a ray and its endpoint is $\tilde{\xi}\in Fix_+(\widehat{\Phi})$ for some principal lift $\widehat{\Phi}$.
  Moreover, if $E$ is an edge in some EG strata then the weak accumulation set of $\xi$ is the unique attracting lamination
  associated to EG strata which contains $E$.

  Conversely, every point $\tilde{\xi}\in Fix_+(\widehat{\Phi})$ is obtained by iteration of a some nonfixed
  edge with a fixed initial vertex.
 \end{lemma}

 Note that in this case $\widehat{\Phi}$ may not be a principal lift and the ray may not be a singular ray as we have
 described it here. What we are really interested in is the conclusion in the last sentence about accumulation set about
 the endpoint. The fact that the weak closure of any attracting fixed point contains an attracting lamination is used
 several times in this paper. In particular note that if we have a conjugacy class that gets attracted to a singular line under
 iterates of $\phi$, then it is also weakly attracted to the attracting laminations which appear in the closure of the endpoints of this
 singular line.

The following lemmas are applications of the above result which, put together, essentially tell that iteration of any edge whose
initial vertex is principal will give us a ray whose endpoint is an attracting fixed point.

\begin{lemma}\label{lemma2}

  Let $v$ be a principal vertex in a CT $f: G\longrightarrow G$ representing a rotationless $\phi\in \out$.
  Then for every direction $d$ originating from $v$,  $f^N_\# (d)$ is a fixed direction for some $N>0$ .

 \end{lemma}

 \begin{proof}
  The number of directions originating at any given vertex is finite. If a direction is not periodic then upon iteration of
  $f_\#$ it must repeat some direction after $N'$ iterations. But since $f$ is
  rotationless every periodic direction has period one. So, $N'$ is bounded by valence of $v$.
  Repeat this for every direction and taking the least common multiple we get $N$ such that $f^N_\#(d)$ has a fixed
  direction for any direction $d$.
 \end{proof}

 \begin{lemma}\label{lemma3}
  Let $v$ be a principal vertex in a CT $f: G\longrightarrow G$ representing a rotationless $\phi\in \out$ and $E'$
  be a non-fixed EG edge originating at $v$. Then $f^k_\#(E')$ weakly accumulates on a singular ray as
  $k\rightarrow \infty$.


 \end{lemma}

 \begin{proof}
  By using the previous lemma we know that the initial direction of $f^k_\#(E')$ is fixed for all $k\geq N$,
  where $N$ is as in Lemma \ref{lemma3}.
  Hence we can write $f^k_\#(E') = Eu$, where $E$ is a principal edge. Since the initial direction of $E$ is a
  fixed direction, we have $f_\#(E) = E\cdotp u'$. So, $f_\#(Eu) = E\cdotp u' f_\#(u)$. Because, if $f_\#(u)$ cancels
  out all of $u'$, then we can write $f_\#(u) = \overline{u'}u''$ and then in $u''$ having an $\overline{E}$ in the beginning
  must be followed
  by $E$, since the initial direction of $f^k_\#(E')$ is principal. This is not possible by definition of $f_\#$. Thus,
  the splitting is indeed there. So, $f_\#^{k+1}(E') = E\cdotp u'f_\#(u)$.
  Upon iteration we see that this ray $R'$ accumulates on a singular ray $R$ with initial direction $E$.

  Next we note that since every finite subpath of $R'$ occurs as a subpath of $R$. This is because each tile of $E'$
is contained in a tile of $E$. Hence the weak closure of $R'$ is contained in the weak closure of $R$. But the weak closure
of $R$ is just $\Lambda^+$. Hence weak closure of $R'$ is the same set and the two rays are asymptotic.

 \end{proof}

\textbf{Structure of Singular lines:} We now state a lemma which is a collection of lemmas from the subgroup decomposition work of
Handel and Mosher that tells us the structure of singular lines and guarantees that one of the leaves of any attracting
lamination is a singular line.

\begin{lemma}\cite[Lemma 3.5, Lemma 3.6]{HM-09}, \cite[Lemma 1.63]{HM-13a}\label{structure}

Let $\phi\in\out$ be rotationless and hyperbolic and $l\in \mathcal{B}_{sing}(\phi)$ then the following are true:
\begin{enumerate}
 \item $l = \overline{R}\alpha R'$ where $R$ for some singular rays $R \neq R'$ and some path $\alpha$ which is
 either trivial or a Nielsen path. Conversely, any such line is a singular line.
\item If $\Lambda\in\mathcal{L(\phi)}$ then there exists a leaf of $\Lambda$ which is a singular line and one of its
ends is dense in $\Lambda$.

\end{enumerate}

\end{lemma}

We include a short sketch of the proof here for sake of completeness, since this result is fundamental to
the work done in this paper here. For more technical details please refer to the original proof.

\begin{proof} Let $f: G\rightarrow G$ be a  CT representing $\phi$.\\
 \textit{Sketch of proof for} (1):
 Suppose $l$ is a singular line, then there exists a principal lift $\Phi$ of $\phi$ which fixes the
 endpoints of this line, which are attracting fixed points. Use Lemma \ref{lemma4} to conclude that there are
 singular rays $R', R$ which converge to these endpoints. Join the initial endpoints of $R, R'$  by
 a path $\tilde{\alpha}$. Then the endpoints of $\tilde{\alpha}$ are fixed and the projection to $G$, say $\alpha$, is a Nielsen path.
 However, the line $\overline{R} \alpha R'$ may not be locally injective at endpoints of $\alpha$. But
 using the lemma that $\alpha$ has a unique decomposition into concatenation of indivisible Nielsen paths
 and fixed edges, one can choose $R, R'$ such that it is locally injective as is done in the original proof in 
 \cite[Lemma 3.5]{HM-09}.
 The converse part follows straight from definition of a singular line.
 This completes the sketch of proof for (1).\\

 \textit{Sketch of proof for} (2) \cite[Lemma 3.6]{HM-09}: Let $H_r$ be the unique EG strata corresponding to $\Lambda$.
 Using the lemma that every generic leaf has a complete splitting (since it can be
 written as a increasing union of tiles) one concludes that a generic leaf $l$ has a splitting of the form
 $\overline{R} \cdot E \cdot R'$, where $E$ is an edge in $H_r$ whose initial vertex is principal and initial
 direction is fixed (existence of $E$ is due to \cite[Lemma 3.19]{FH-11}). Consider a further splitting of
 $l = \overline{R} \cdot E \cdot \alpha \cdot R''$, where $\alpha$ is the longest initial segment of $R'$
 which can be expressed as a concatenation of indivisible Nielsen paths and fixed edges. Recall that there is an
 upper bound on the number of components in the complete splitting of any path which occurs as a concatenation of
 fixed edges and indivisible Nielsen paths. Using this, choose $l$ to maximize the number of components of $\alpha$.
 Now iterate $l$ by $f_\#$ to get a leaf (since $\Lambda$ is invariant under $f_\#$) and this leaf will be a
 singular line due to choice of the splitting above.

\end{proof}

\textbf{Finiteness of singular lines: }
Before we end this overview about singular lines we prove a lemma that shows that $\mathcal{B}_{sing}(\phi)$ is a finite set.
We will need this result to prove the uniform finiteness of fibers of Cannon-Thurston maps.

Recall that the normal subgroup Inn($\mathbb{F}$) acts on $\aut$ by conjugation and the orbits of this action define
an equivalence relation on $\aut$ called \textit{isogredience}. The set of principal automorphisms representing some
$\phi\in\out$ is invariant under isogredience. Also any two elements of $\aut$ which are in the same
equivalence class have the same attracting fixed points at the boundary. \cite[Remark 3.9]{FH-11} shows
that the number of isogredience classes of principal automorphisms are finite.

Let $\mathcal{B}_{\mathsf{Fix}_+}(\phi)$ denote the set of lines with endpoints in $\mathsf{Fix}_+(\phi)$
\begin{prop}\label{singfin}
 For any hyperbolic $\phi\in\out$ the set of lines $\mathcal{B}_{\mathsf{Fix}_+}(\phi)$ is a finite set and its cardinality is bounded above
 by a number depending only on rank$(\mathbb{F})$.

\end{prop}
The author thanks Lee Mosher for help with the proof.

\begin{proof} Let rank$(\mathbb{F}) = N$.
 There are only finitely many isogredience classes of principal automorphisms in $P(\phi)$. Pick representatives
 $\Phi_1, \Phi_2,..., \Phi_K$ as principal automorphisms that represent these isogredience classes.
 Since $\phi$ is hyperbolic $\mathsf{Fix}_N(\widehat{\Phi}_i) = \mathsf{Fix}_+(\widehat{\Phi}_i)$ for
 all $i = 1, 2,.... , K$. If $a_i = |\mathsf{Fix}_+(\widehat{\Phi}_i)|$ for $i=1, 2, .... , K$ then
 the main inequality of counting indexes in \cite[Theorem 1']{GJLL-98} gives us $$a_1 + a_2 + ... + a_K \leq 2N$$

 Since any line in $\mathcal{B}_{sing}(\phi)$ is obtained by joining two points in $\mathsf{Fix}_+(\widehat{\Phi}_i)$ for some 
 principal lift $\widehat{\Phi}_i$, the
 above inequality shows that maximum number of possible  lines in $\mathcal{B}_{sing}(\phi)$ is uniformly bounded above by some function of $N$.

 Now, coming to $\mathcal{B}_{\mathsf{Fix}_+}(\phi)$, by using Lemma \ref{lemma4} we can deduce that the  endpoints of these lines  
 are ones that are defined by singular rays. So, $\mathcal{B}_{\mathsf{Fix}_+}(\phi)$ is obtained by joining 
 endpoints of lines in $\mathcal{B}_{sing}(\phi)$, hence the cardinality of $\mathcal{B}_{\mathsf{Fix}_+}(\phi)$ is uniformly 
 also bounded above by some number depending on rank$(\mathbb{F})$.
\end{proof}

\subsection{Weak attraction theorem }
\label{sec:9}

The following lemma is very important for our purposes. The form in which we will be using it  is: given a line $l$ such that
one endpoint of $l$, $\xi$ say, is carried by some $\mathcal{A}_{na}(\Lambda^\pm_\phi)$ (i.e. there exists some $[A]\in\mathcal{A}_{na}(\Lambda^\pm_\phi)$
such that $\xi\in\partial A$) but the other point of $l$ is not in $\partial A$ then the lemma tells us that the line
gets attracted to either the attracting lamination or the repelling lamination (or both). We call it the Weak attraction theorem.

\begin{lemma}[\cite{HM-13c} Corollary 2.17, Theorem H]
\label{WAT}
 Let $\phi\in \out$ be a rotationless and exponentially growing. Let $\Lambda^\pm_\phi$ be a dual lamination pair for $\phi$. Then for any line $\gamma\in\mathcal{B}$ not carried by $\mathcal{A}_{na}(\Lambda^{\pm}_\phi)$ at least one of the following hold:
\begin{enumerate}
 \item $\gamma$ is attracted to $\Lambda^+_\phi$ under iterations of $\phi$.
   \item $\gamma$ is attracted to $\Lambda^-_\phi$ under iterations of $\phi^{-1}$.
\end{enumerate}
Moreover, if $V^+_\phi$ and $V^-_\phi$ are attracting neighborhoods for the laminations $\Lambda^+_\phi$ and $\Lambda^-_\phi$ respectively, there exists an integer $l\geq0$ such that at least one of the following holds:
\begin{itemize}
 \item $\gamma\in V^-_\phi$.
\item $\phi^l(\gamma)\in V^+_\phi$
\item $\gamma$ is carried by $\mathcal{A}_{na}(\Lambda^{\pm}_\phi)$.
\end{itemize}

\end{lemma}

\begin{lemma}
 Suppose $\phi,\psi\in \out$ are two exponentially growing automorphisms with attracting laminations $\Lambda^+_\phi$ and $\Lambda^+_\psi$, respectively. 
 If a generic leaf $\lambda\in\Lambda^+_\phi$ is in $\mathcal{B}_{na}(\Lambda^+_\psi)$ then the whole lamination $\Lambda^+_\phi\subset \mathcal{B}_{na}(\Lambda^+_\psi)$.

\end{lemma}

\begin{proof}
 Recall that a generic leaf is birecurrent. Hence, $\lambda\in \mathcal{B}_{na}(\Lambda^+_\psi)$ implies that $\lambda$ is either carried by $\mathcal{A}_{na}(\Lambda^\pm_\psi)$ or it is a generic leaf of some element of $\mathcal{L}(\psi^{-1})$.
 First assume that $\lambda$ is carried by $\mathcal{A}_{na}(\Lambda^\pm_\psi)$. Then using Lemma \ref{NAS} item 1, we can conclude that $\Lambda^+_\phi$ is carried by $\mathcal{A}_{na}(\Lambda^+_\psi)$.

 Alternatively, if $\lambda$ is a generic leaf of some element $\Lambda^-_\psi\in \mathcal{L}(\psi^{-1})$, then the weak closure $\overline{\lambda} = \Lambda^+_\phi=\Lambda^-_\psi$ and we know $\Lambda^-_\psi$ does not get attracted to $\Lambda^+_\psi$.
 Hence, $\Lambda^+_\phi\subset\mathcal{B}_{na}(\Lambda^+_\psi)$.
\end{proof}

\begin{lemma}\label{circuitatt}
 Let $\phi\in\out$ be rotationless and $[c]$ be some conjugacy class in $\mathbb{F}$. If $\Lambda^-$ is a
 repelling lamination for $\phi$ and $V^-$ is an attracting neighborhood for $\Lambda^-$, then there are only finitely
 many values of $k>0$ such that $\phi^k_\#([c]) \in V^-$ .
\end{lemma}

\begin{proof}
 Since $\phi$ is rotationless, $\Lambda^-$ is $\phi$ invariant and so $\phi^{-1}_\#(V^-) \subset V^-$ .
 Since any weak neighborhood can be defined
 by some finite subpath of a generic leaf (\cite[Corollary 3.1.11]{BFH-00}), we see that under iterates of $\phi^{-1}_\#$ this subpath
 grows exponentially  and since by \cite[Lemma 3.1.16]{BFH-00} a cicuit cannot be a generic leaf,
 we have that there exists some $K$ such that $[c]\notin \phi^{-K}_\#(V^-)$. Hence
 for every $k\geq K$, $\phi^k_\#([c])\notin V^-$.

\end{proof}


 \section{Main Theorem}
 \label{main}
 In this section the goal is to list all the possible weak limits of $\phi^k([c])$ as $k\rightarrow\infty$, where $\phi$ is a
 hyperbolic outer automorphism of $\mathbb{F}$ and $[c]$ is any conjugacy class in $\mathbb{F}$. Let us denote this collection of
 weak limits by $\mathcal{WL}(\phi)$. We shall see from Theorem \ref{Limit} that
 $$\mathcal{WL}(\phi) =  \mathcal{B}_{gen}(\phi) \cup \mathcal{B}_{sing}(\phi)$$

 The set on the right hand side of the equality is a very well understood collection of lines in the theory of $\out$. Recall that it appears in the description of the set of nonattracted lines in \ref{concat} and this shows that the dynamics
 of conjugacy classes under iterates of a hyperbolic outer automorphism is very well controlled. There is an immediate observation one can make here; if a line $\gamma$ is a weak limit of some conjugacy class under iterates of a hyperbolic $\phi$, then $\gamma$ is not weakly attracted to any repelling lamination of $\phi$.

  We begin this section
 with a lemma that, in some sense, is the heart of the proof here. It does a detailed analysis of the components of any complete splitting of a circuit under iterates of a hyperbolic outer automorphism. A clear understanding of the proof of this lemma gives a very clear indication of why the above equality is what one should expect.

 \textbf{Notation:} EG strata is an abbreviation for ``exponentially growing strata'', NEG strata for ``non-exponentially growing strata ''
  , CT stands for ``completely split relative train track maps'', $\mathcal{B}_{gen}(\phi)$ denotes the collection of
  all generic leaves of all attracting laminations for $\phi$ and $\mathcal{B}_{sing}(\phi)$ denotes the collection of
  all singular lines of $\phi$.

 \begin{lemma}\label{EG}
 Let $\phi\in\out$ be a rotationless and hyperbolic and $f:G\longrightarrow G$ be any CT representing $\phi$. If $\sigma$ is a circuit in $G$
 then the terms that appear in a complete splitting of $\sigma$ must contain an EG edge and hence is
 exponentially growing. Moreover, $\sigma$ is attracted to some element of $\mathcal{L}(\phi)$.
 \end{lemma}

 \begin{proof} By using Lemma \ref{circuitsplit} we know that $f^k_\#(\sigma)$ is completely split for all sufficiently large
 $k$.

  Existence of linear-NEG strata and
  geometric strata in $G$ automatically guarantees the existence of a conjugacy class that is fixed by $\phi$ which is not possible
  since $\phi$ is hyperbolic. Hence no term in any complete splitting of $f^k_\#(\sigma)$ can be a linear edge or
  exceptional path or edge in a geometric EG strata.

  If $G$ has any superlinear-NEG strata, then consider $E'=H_r$ to be the non-fixed  NEG stratum of lowest height.
  By (NEG edges, \ref{CT}) there exists a circuit $u_r \subset G_{r-1}$ such that $f_\#(E')= E'\cdotp u_r$ is a splitting and
  $u_r$ is completely split. The terms of its complete splitting are edges in irreducible stratum, indivisible Nielsen paths,
   fixed edges or maximal subpaths in zero strata which are taken. We claim that some term of such a splitting must be an EG edge.

  No term in splitting of $u_r$ can be a linear NEG edge since
  $\phi$ is hyperbolic. If we have taken, maximal subpaths in zero strata then we are done because one needs to pass through an
  EG stratum to enter and exit a zero strata (Envelope, \ref{CT}) and the only way to do this would be if we had EG edges appearing as terms of 
  complete splitting of $u_r$ (on either side of the splitting component contained in the zero strata) since all zero strata are contractible
  and EG strata that contain indivisible Nielsen path do not envelope zero strata (see Lemma \ref{propzero}) and 
  indivisible Nielsen paths of higher EG strata have both their endpoints in the same strata by Lemma \ref{np}.
  So if we do not have paths contained in some zero strata then the only remaining possibilities for the components of complete splitting of $u_r$ remaining  are that of 
  EG edges, fixed edges and
  indivisible Nielsen paths (since we have assumed that $E_r$ has lowest height among NEG edges). If we have only fixed edges and indivisible Nielsen paths appearing as terms of 
  complete splitting of $u_r$ then it would imply implies $u_r$ is a closed Nielsen path, which is  not true since
  $E'$ is superlinear NEG strata. Hence some term of complete splitting of $u_r$ must be an EG edge. Now for superlinear NEG edges of height  $s>r$, one 
  can similarly show by an induction argument that some term in the complete splitting of $f^k(u_s)$ must be an EG edge, for some $k>0$.

  Since we have ruled out the possibility of any linear NEG edge and geometric strata in $G$, the only choices we have
  in the complete splitting of $f^k_\#(\sigma)$ are EG edges, superlinear NEG edge, fixed edges, indivisible Nielsen paths and maximal connecting
  subpaths in a zero strata which are taken. Presence of only fixed edges and indivisible Nielsen paths
  would imply existence of a periodic circuit and
  violate the hyperbolicity assumption on $\phi$. So we must have at least an EG edge, superlinear NEG edge or a taken connecting path in a zero stratum. 
  If the last case happens then, by the same argument we gave for superlinear-NEG edges, on either side of this
  subpath of zero strata the terms of complete splitting must be EG edges. Also, if we have a superlinear NEG edge, then by same argument above, an EG edge appears as a term of 
  complete splitting in $f^t_\#(\sigma)$, for some $t>0$ . Hence some term in the complete splitting of
  $f^k_\#(\sigma)$ must be an EG edge.

  Thus $\sigma$ is exponentially growing
  and is weakly attracted to the attracting laminations which are associated to the EG strata whose edges appear
  in some complete splitting of $f^k_\#(\sigma)$ for some $k>0$.

 \end{proof}

 Notice that the above lemma implies that given any conjugacy class $[c]$ in $\mathbb{F}$, there exists at least one 
 one nonattracting subgroup system $\mathcal{A}_{na}(\Lambda^+_\phi)$ which does not carry $[c]$. We shall see shortly in Lemma \ref{lemma1} that 
 this behaviour is also observed in any line that occurs as a weak limit of $[c]$ under iterates of $\phi$.

 \begin{lemma}\label{fixlam}
  Let $\phi\in\out$ be rotationless and hyperbolic.
  Then the weak closure of every point in $\xi\in\mathsf{Fix}_+(\phi)$ contains an element of
  $\Lambda\in\mathcal{L}(\phi)$. Moreover, if the principal edge whose  iteration generates the singular ray
  with endpoint $\xi$ is an EG edge then 
  $\Lambda$ is the unique attracting lamination associated to the EG strata that contains the principal edge.
 \end{lemma}

 \begin{proof}
 Let $f:G\longrightarrow G$ be CT representing $\phi$.
 Given any point $\xi\in\mathsf{Fix}_+(\phi)$, it is obtained by iteration of an edge $E$ whose initial vertex
 is principal and initial direction is fixed (such an edge is called \textit{principal edge}) (use Lemma \ref{circuitsplit}).
 If $E$ is an EG edge, and $k\rightarrow\infty$, $f^k_\#(E)$ generates a singular ray with endpoint in $\xi$ then the closure of $\xi$ equals the 
 unique attracting lamination associated to the EG strata that contains $E$. 
 
On the other hand if $E$ is superlinear NEG edge then $f_\#(E)= E\cdotp u$ is a splitting and the circuit $u$ is completely split. Arguing just as we did in the previous lemma, 
some term in the complete splitting of $f^k_\#(u)$ will be an EG edge and hence the ray generated by $f^k_\#(E)$, as $k\to\infty$, is a singular ray which weakly accumulates on some 
attracting lamination $\Lambda\in\mathcal{L}(\phi)$ (where $\Lambda$ is the unique attracting lamination associated to the EG strata whose edge appears in some complete splitting of $f^k_\#(u)$). 
This implies that the weak closure of $\xi$ contains $\Lambda$.

 This shows that the weak closure of $\xi$ is exactly the unique attracting lamination associated to the strata that
 contains the edge $E$ by using Lemma \ref{lemma4}.

 \end{proof}

Our next goal is to inspect the dynamical nature of a line $\gamma$ which appears as a weak limit of some conjugacy class 
under iterates of some hyperbolic outer automorphism $\phi$. We shall prove in Lemma \ref{coro} and Lemma \ref{lemma1} that 
any such line gets attracted to some attracting lamination of $\phi$ (under iterates of $\phi$) but is never attracted to any repelling 
lamination of $\phi$ (under iterates of $\phi^{-1}$). Notice that this implies $\gamma\in\mathcal{B}_{na}(\Lambda^-_\phi)$ for every 
repelling lamination of $\phi$ and $\gamma$ is not carried by $\mathcal{A}_{na}(\Lambda^+_\phi)$ for at least 
one attracting lamination of $\phi$.

\begin{lemma}\label{coro}
  Let $\phi$ be a rotationless hyperbolic outer automorphism and $f:G\longrightarrow G$ be a CT
 representing $\phi$. Suppose $\sigma$ is a circuit in $G$.
 If $\gamma$ is a weak limit of $\sigma$ under action of $f_\#$, then $\gamma$ is not attracted to any
 element of $\mathcal{L}(\phi^{-1})$.
 \end{lemma}

 \begin{proof}

  Suppose on the contrary $\gamma$ gets attracted to some $\Lambda^-$. If $V^-$ is an attracting neighborhood  of $\Lambda^-$ then
  $\phi^{-t}_\#(\gamma)\in V^-$ for some $t\geq 1$. This implies that $\gamma\in \phi^t_\#(V^-)$. Since
  $\phi^t_\#(\Lambda^-) = \Lambda^-$ for all $t\geq 1$ we can conclude that $\gamma$ contains some subpath $\alpha$
  of a generic leaf of $\Lambda^-$. So $\alpha$ is contained in $f^k_\#(\sigma)$ as a subpath for all sufficiently large $k's$.

  But this would mean that $f^k_\#(\sigma)\in V^-$ for infinitely many value of $k$, which contradicts the lemma that
  $\sigma$ cannot be attracted to $\Lambda^-$ by iterates of $f^k_\#$ (Lemma \ref{circuitatt}).

 \end{proof}

 \begin{cor}\label{coro1}
  Let $\phi\in \out$ be rotationless and hyperbolic. If $\gamma$ is line which is a weak limit of some conjugacy class
  in $\mathbb{F}$ then no endpoint of any lift of 
 $\gamma$ is in $\mathsf{Fix}_-(\phi)$.
 \end{cor}

 \begin{proof} Suppose $\gamma$ had an endpoint  $\xi\in\mathsf{Fix}_-(\phi)$. Since $\phi$ is hyperbolic, so is $\phi^{-1}$ and using this we derive a contradiction to
 Lemma \ref{coro}.

  Choose a CT $f':G'\longrightarrow G'$ representing some rotationless power of $\phi^{-1}$.
  Using Lemma \ref{lemma4}, every point in $\mathsf{Fix}_-(\phi)$ is obtained by iteration of an edge $E'$ whose initial vertex is principal and
  initial direction is fixed. By using
  Lemma \ref{fixlam} we get that the weak closure of $\xi$, and hence $\gamma$,
  will contain an attracting lamination $\Lambda^-$ associated to the EG strata $E'$. This means $\gamma$
  is attracted to $\Lambda^-$. Which contradicts Lemma \ref{coro}.

 \end{proof}

 \begin{lemma}\label{lemma1}
  Let $\phi$ be a hyperbolic automorphism and suppose $\phi^k(c)$ converges weakly to some line  $\gamma$ as
  $k\rightarrow \infty$. 
  Then $\gamma$ is weakly attracted to at least one element of $\mathcal{L}(\phi)$.
 \end{lemma}

 \begin{proof}
  Replacing $\phi$ by some $\phi^N$ if necessary, we may assume that $\phi$ is rotationless. Let $f:G\longrightarrow G$
  be a CT representing $\phi$ and $\sigma$ be the realization of the conjugacy class $c$ in $G$.
  Note that if $\gamma\in \mathcal{B}_{gen}(\phi)\cup\mathcal{B}_{sing}(\phi)$ then weak closure of $\gamma$ contains
  at least one attracting lamination of $\phi$ and hence by definition, $\gamma$ is attracted to that attracting lamination.
  Hence we are left with the case when $\gamma\notin \mathcal{B}_{gen}(\phi)\cup\mathcal{B}_{sing}(\phi) $.
  
  By using  Corollary \ref{coro1} we know that no endpoint of any lift of $\gamma$ is in $\mathsf{Fix}_-(\phi)$. 
  Also Lemma \ref{coro} tells us that $\gamma$ is not a generic leaf of some repelling lamination of $\phi$. Hence using 
  Lemma \ref{concat} we can conclude that if $\gamma$ is not attracted to some attracting lamination $\Lambda^+$ then it must be carried 
  by $\mathcal{A}_{na}(\Lambda^+)$. This implies that theres exists some lift $\tilde{\gamma}$ of $\gamma$ such that 
  both endpoints of $\tilde{\gamma}$ are carried by $\partial A$ for some $[A]\in\mathcal{A}_{na}(\Lambda^+)$.

  Therefore we can conclude that if $\gamma$ is not attracted to any element of $\mathcal{L}(\phi)$ then is carried by 
  $\mathcal{A}_{na}(\Lambda^+_i)$ for every element $\Lambda^+_i\in\mathcal{L}(\phi)$ where $i=1, 2, ..., K$ for some $K<\infty$.
  So there are 
  $[A_i]\in\mathcal{A}_{na}(\Lambda^+_i)$ such that $\partial A_i$  carries both endpoints of some lift of $\gamma$.
  Hence $\cap\partial A_i\neq \emptyset$. Thus $\cap A_i\neq\emptyset$ and so there exists a conjugacy class $[c]$ which is 
  carried by $\mathcal{A}_{na}(\Lambda^+_i)$ for every $i$. But this contradicts Lemma \ref{EG}. Therefore $\gamma$ 
  must be attracted to some element of $\mathcal{L}(\phi)$.

 \end{proof}

 The following proposition gives us one side of the inclusion of Theorem \ref{Limit}.

 \begin{prop}\label{nongen0}
  Let $f:G\longrightarrow G$ be a CT representing a rotationless hyperbolic element of $\out$ and $\sigma$ is a circuit in
  $G$. If $l$ is a weak limit of $\sigma$ under the action of $f_\#$ which is  not a singular line for $\phi$, then either $l$ must be generic leaf of some element of
  $\mathcal{L}(\phi)$ or both endpoints of $l$ must be in $\mathsf{Fix}_+(\phi)$. 
 \end{prop}

 \begin{proof}
 Suppose on the contrary that $l\notin \mathcal{B}_{gen}(\phi)$.
 Given that $l$ is not  a singular line, and since
 we know that $l$ is not weakly attracted to any element of $\mathcal{L}(\phi^{-1})$ we can apply
 Lemma \ref{concat} and deduce that $l\in \mathcal{B}_{ext}(\Lambda^\pm, \phi)$ for every dual lamination pair
 of $\phi$ since $l\notin \mathcal{B}_{gen}(\phi) \cup \mathcal{B}_{sing}(\phi)$. But there exists at least one element
  $\Lambda^+\in \mathcal{L}(\phi)$,  to which $l$ is attracted (by Lemma \ref{lemma1}), which implies that $l$ is not carried by
 $\mathcal{A}_{na}(\Lambda^\pm)$. If we denote the two distinct endpoints of $l$ by $P$ and $Q$, then at least 
 one endpoint, $P$ say, is in $\mathsf{Fix}_+(\phi)$. We now procced to give an argument by contradiction.
 
 Suppose that $P\in\mathsf{Fix}_+(\phi)$ but $Q\notin \mathsf{Fix}_+(\phi)$. 
 In this case 
 notice that $Q\notin \mathsf{Fix}_-(\phi)$, since that would violate corollary \ref{coro1}. 
Since $Q\notin \mathsf{Fix}_+(\phi)$, iteration of $Q$ by $\phi^{-t}_\#$ converges to a point $Q^-$ in
$\mathsf{Fix}_-(\phi^{-t})$ (for some $t\geq 1)$ by Lemma \ref{attfix}. But this implies $l$ is attracted to some $\Lambda^-_{Q^-}$ (since $\mathsf{Fix}_+(\phi)$ and
$\mathsf{Fix}_-(\phi)$ are disjoint and hence $P\neq Q^-$), where $\Lambda^-_{Q^-}$ is in the weak closure
of $Q^-$ (by using Lemma \ref{fixlam}). But this contradicts Lemma \ref{coro1}. Hence this case is not possible.

 Therefore $P\in\mathsf{Fix}_+(\phi)$ and $Q\in \mathsf{Fix}_+(\phi)$ and we get the desired conclusion.
 
 Note that since we have assumed that $l$ is not a singular line there does not exist a principal lift 
 that fixes both $P$ and $Q$, hence this case covers more than singular lines.

 \end{proof}

 We now state an important lemma from the Subgroup Decomposition work of Handel and Mosher. This lemma can
 be used to conclude that every line in $\mathcal{B}_{sing}(\phi)$ and every leaf in any
 attracting lamination for $\phi$  occurs as a weak limit of some conjugacy class under action of $\phi_\#$.

 \begin{lemma}\label{converse}
  \cite[Lemma 1.52]{HM-13a} For each $P\in \mathsf{Fix}_+(\phi)$ there is a conjugacy class $[a]$ which is
  weakly attracted to every line in the weak accumulation set of $P$.
 \end{lemma}

\begin{cor}\label{nongen}
 Let $\phi$ be a rotationless hyperbolic outer automorphism and $\Lambda\in \mathcal{L}(\phi)$ be an attracting
 lamination. Then every nongeneric leaf (if it exists) of $\Lambda$ must be a line with both endpoints 
 in $\mathsf{Fix}_+(\phi)$.

 \end{cor}

 \begin{proof}

 Recall that Lemma \ref{structure} says that every attracting lamination $\Lambda\in\mathcal{L}(\phi)$ contains a singular line as a leaf
 , one of whose ends is dense in $\Lambda$. Lemma \ref{converse} tells us that there exists a conjugacy class
 $[a]$ which is weakly attracted to every line in $\Lambda$.

  The proof now follows directly from Proposition \ref{nongen0}.
 \end{proof}

\begin{remark}
 It is worth noting that this corollary is very special to hyperbolic outer automorphisms and will fail if $\phi$ is
 not hyperbolic. Once we have linear NEG edges or geometric strata the extended boundary takes a much more complex structure
 and one can easily construct examples where this lemma will fail.
\end{remark}

We finally state and prove the main theorem of this section which classifies all weak limits of conjugacy classes
 under iterations of a hyperbolic outer automorphism. 
 
 \textbf{Notation:} Let $\mathcal{B}_{\mathsf{Fix}_+}(\phi)$ denote the set of all 
 lines with endpoints in $\mathsf{Fix}_+(\phi)$. Note that 
 $\mathcal{B}_{sing}(\phi)\subset \mathcal{B}_{\mathsf{Fix}_+}(\phi)$.

 \begin{thm}
  \label{Limit}

 Suppose $\phi \in \out$ is a hyperbolic outer automorphism and $[c]$ is
 any conjugacy class in $\mathbb{F}$. Then the weak limits of $[c]$ under iterates of $\phi$ are in
 $$\mathcal{WL}(\phi) : =  \mathcal{B}_{gen}(\phi) \cup \mathcal{B}_{\mathsf{Fix}_+}(\phi)$$

 Conversely, any line in $\mathcal{WL}(\phi)$ is a weak limit of some conjugacy class in $\mathbb{F}$ under action of $\phi$.

 \end{thm}

 \begin{proof}
 We may assume without loss that $\phi$ is rotationless, since the work of Feighn and Handel
in \cite{FH-11} show that there exists some number $K$ such that $\phi^K$ is rotationless for any $\phi\in\out$
and it is obvious that the set of weak limits of conjugacy classes is invariant under passing to a finite power.

 If $\gamma$ occurs as a weak limit of some conjugacy class under action of $\phi_\#$ Proposition \ref{nongen0}
 guarantees that it is in $\mathcal{B}_{gen}(\phi) \cup \mathcal{B}_{\mathsf{Fix}_+}(\phi)$.

 The proof of the converse part
  directly follows by using Lemma \ref{fixlam} and  Lemma \ref{converse}, since every attracting lamination $\Lambda$ of $\phi$ has a leaf
 which is a singular line and one of whose ends is dense in $\Lambda$ by using Lemma \ref{structure}.

  \end{proof}
\begin{remark}
  Note that the situation is much simpler in case of fully irreducible and hyperbolic $\phi$ since the unique attracting lamination 
  of such an element does not contain any nongeneric leaves. Hence the equality just reduces to 
  $\mathcal{WL}(\phi) : =  \mathcal{B}_{gen}(\phi) \cup \mathcal{B}_{sing}(\phi)$. 
  \end{remark}
  
We end this section with a corollary which characterizes hyperbolicity in terms of weak limits of conjugacy classes.

\begin{cor}\label{althyp}
 Let $\phi\in\out$ be rotationless. Then $\phi$ is hyperbolic if and only if every
 conjugacy class $[c]$ is weakly attracted to some line in $\mathcal{B}_{sing}(\phi)$.
\end{cor}

\begin{proof}
 Follows directly from Theorem \ref{Limit} and Lemma \ref{fixlam}.
\end{proof}

 \section{Applications}
 \label{appl}
 In this section we will look at some applications of the results we proved in the previous section.
 The first half of this section deals with Cannon-Thurston maps for a hyperbolic $\phi\in\out$ and
 quasiconvexity of infinite index, finitely generated subgroups of $\mathbb{F}$ in the extension group
 $G= \mathbb{F} \rtimes_\phi \mathbb{Z}$. In this half of the section we carefully develop the ideas and comment on
 possible motivations which leads to the Theorems \ref{ct1}, \ref{qc1}, \ref{qc2}.

 The next half of the section generalizes the results we prove in the first half to the case when
 we replace $\phi$ by a Gromov-hyperbolic and purely \textit{atoroidal} group (recall that in such a group
 every element is hyperbolic). Except for description of the ending lamination set
 we will generally be brief about the proofs since they are almost identical to the ones in the first half of this section.


\begin{remark}\label{flip}
There are a couple of important points that we would like to clarify before we proceed with the applications. This is
for benefit of readers who are not familiar with standard terminologies in the weak attraction language.
 \begin{enumerate}
  \item A ``line'' $l\in \widetilde{\mathcal{B}}$ is not just a geodesic in $\mathbb{F}$ joining two points in
  $\partial\mathbb{F}\times\partial\mathbb{F}$. Recall that
  $\widetilde{\mathcal{B}}=\{ \partial \mathbb{F} \times \partial \mathbb{F} - \vartriangle \}/(\mathbb{Z}_2)$ (see Preliminaries \ref{sec:2}).
  The action of $\mathbb{Z}_2$ on $\partial\mathbb{F}\times\partial\mathbb{F}$ is by interchanging endpoints.
  So a line of $\widetilde{\mathcal{B}}$ is unoriented and flip-invariant.
  $\widetilde{\mathcal{B}}$ carries the weak topology induced from Cantor topology on
  $\partial \mathbb{F}$.
    Elements of $\mathcal{B}=\widetilde{\mathcal{B}}/\mathbb{F}$, are projections of line in $\widetilde{\mathcal{B}}$ in the
    quotient space which is compact but not Hausdorff. Elements of $\mathcal{B}$ are also referred to as lines.

  \item When we say $\phi^k_\#([c])$ converges to a line $l$ or $l$ is a weak limit of $\phi^k_\#([c])$
  as $k\rightarrow\infty$ it is equivalent to saying every subpath $\alpha$ of $l$ occurs as subpath of $\phi^s_\#([c])$ for all
  sufficiently large $s$.

  Lifting to $\widetilde{\mathcal{B}}$, if $\tilde{\alpha}$ is a subpath of $\tilde{l}$ then either $\tilde{\alpha}$ or $\tilde{\alpha}^{-1}$
  occurs as a subword of some cyclic permutation of a word $w$  representing $\phi^s_\#([c])$.
    To summarize, $\phi^k_\#([c])$ converges to a line $l$ or $l$ is a weak limit of $\phi^k_\#([c])$
  as $k\rightarrow\infty$ is equivalent to saying that for every lift $\tilde{l}\in\widetilde{\mathcal{B}}$ and every subword
  $\tilde{\alpha}$ of $\tilde{l}$, either $\tilde{\alpha}$ or $\tilde{\alpha}^{-1}$ occurs as a subword of some cyclic permutation of
  a word $w_s$ representing $\phi^s_\#([c])$ for all sufficiently large $s$.

 \end{enumerate}

\end{remark}

\textbf{Notation: } By $\widetilde{\mathcal{WL}}(\phi)$ we denote all the lifts of lines in $\mathcal{WL}(\phi)$ to $\mathcal{B}$.

 \subsection{Canon-Thurston maps for hyperbolic $\phi$.}
 \label{sec:A1}
Let $\Gamma$ be a word-hyperbolic group and $H<\Gamma$ is a word-hyperbolic subgroup. If the inclusion map $i: H \rightarrow \Gamma$ extends to a
 continuous map of the boundaries $\widehat{i}: \partial H\rightarrow \partial \Gamma$ then
$\widehat{i}$ is called a Cannon-Thurston map. When it does exist, it is an interesting question to know what its properties are.
Its precise behavior is captured by the notion of \textit{Ending laminations}.
The original definition was given in \cite{Mj-97} and for Free groups it was
later modified and used in \cite{KL-13} by Kapovich and Lustig. 
Let $$ 1 \rightarrow \mathbb{F} \rightarrow G \rightarrow \mathcal{H} \rightarrow 1 $$ be an exact sequence of
hyperbolic groups where $\mathcal{H}< \out$. Then Mitra defined:

\begin{definition}\label{el}
 Let $z\in \partial \mathcal{H}$ and $\{\phi_n\}$ be a sequence of vertices on a geodesic joining $1$ to $z$
 Cayley graph of $\mathcal{H}$. Define
 $$\Lambda_{z,[c]} = \{l\in\partial\mathbb{F}\times\partial\mathbb{F} | w \text{ subword of } l\Rightarrow\exists n\ni w \text{ or } w^{-1}\text{ subword of } {\phi_n}_\#([c]) \}$$
 $$ \Lambda_z = \bigcup_{c\in\mathbb{F}-\{1\}} \Lambda_{z,[c]}$$
 $$\Lambda_{\mathcal{H}} = \bigcup_{z\in\partial\mathcal{H}} \Lambda_z$$
\end{definition}

 In \cite[Lemma 3.3]{Mj-97} shows that $\Lambda_z$ is independent of the choice of the
sequence $\phi_n$.

Strictly speaking Mitra's definition is much more general, but we have adapted the definition as the special
case for $\out$.

 The Cannon-Thurston map, when it exists, identifies the endpoints of the certain leaves of ending lamination. Our goal here is to understand the class of
leaves that get identified by this map by using the theory of attracting laminations and singular lines.

Let $\Gamma_\phi$ denote the mapping torus for a hyperbolic $\phi\in\out$.
The precise statement for the behaviour of Cannon-Thurston map is given by:

\begin{thm}\cite[Theorem 4.11]{Mj-98}
 If $\phi\in\out$ is a hyperbolic outer automorphism then the Cannon-Thurston map $\widehat{i}: \partial\mathbb{F} \rightarrow \partial\Gamma_\phi$ exists.
 Moreover, $\widehat{i}(X) = \widehat{i}(Y)$ if and only if
 the line $l\in\partial\mathbb{F}\times \partial\mathbb{F}$ joining $X$ to $Y$ is in $\Lambda_z$ for some $z\in\partial\mathcal{H}$.
\end{thm}

 We now proceed to give the description of $\Lambda_z$ using our work in the earlier section.

If $\phi\in \out$ is hyperbolic then by Brinkmann's work \cite{Br-00} we have an exact sequence
$$ 1 \rightarrow \mathbb{F} \rightarrow G \rightarrow \langle \phi \rangle \rightarrow 1 $$ of hyperbolic groups
where $G= \mathbb{F} \rtimes_\phi \mathbb{Z}$ is the mapping torus of $\phi$

\begin{lemma}\label{el1}
 Let $\mathcal{H} = \langle \phi \rangle$ for some hyperbolic $\phi\in\out$. Then
 $$\Lambda_{\langle \phi \rangle} = \widetilde{\mathcal{WL}}(\phi) \cup \widetilde{\mathcal{WL}}(\phi^{-1})$$

\end{lemma}

\begin{proof}
In Definition \ref{el} if we let $\mathcal{H} = \langle \phi \rangle$  we see that $\partial\mathcal{H}$ has exactly two points, $z_1$ and $z_2$ say, and
the sequences that converge to these points are $\phi^n$ and $\phi^{-n}$ respectively.
If we use our observations in (2) of Remark \ref{flip}, we see that the set of lines in $\Lambda_{z_1, [c]}$ are exactly
the lines $l\in\widetilde{\mathcal{B}}$ such that $[c]$ weakly converges to the projection of $l$ in $\mathcal{B}$
 under iteration of $\phi_\#$.
So, $\Lambda_{z_1} = \widetilde{\mathcal{WL}}(\phi)$. Similarly  $\Lambda_{z_2} = \widetilde{\mathcal{WL}}(\phi^{-1})$.
Thus we have $\Lambda_{\langle \phi \rangle} = \widetilde{\mathcal{WL}}(\phi) \cup \widetilde{\mathcal{WL}}(\phi^{-1})$.

\end{proof}

Kapovich and Lustig in \cite[Theorem 5.4]{KL-13} showed the following:
\begin{thm}
 For a hyperbolic and fully irreducible $\phi\in\out$, the Cannon-Thurston map $\widehat{i}:\partial\mathbb{F}\
 \rightarrow \partial \Gamma_\phi$ is a finite-to-one map and the cardinality of the preimage set of each
 point in $\partial\Gamma_\phi$ is bounded by $2N$, where $N=\text{rank }(\mathbb{F})$.
\end{thm}

We shall improve this theorem by removing the fully irreducible assumption. However the uniform
bound that we provide is not a sharp bound like the one obtained above (see \cite[Remark 5.9]{KL-13}.

We use the following lemma by  \cite[Lemma 2.15]{HM-13c} to prove proposition \ref{asym2}, which is the main technical
result in this section that tells us exactly which elements of $\mathcal{WL}(\phi)$ are identified. A baby version of the lemma first
appeared in \cite[Lemma 3.3]{HM-11} for the fully irreducible case.

\begin{lemma} \label{asym1}
 For any rotationless $\phi$ and generic leaves $l', l''$ of $\phi$, if some end of $l'$ is asymptotic to some
 end of $l''$, then $l', l'' \in\mathcal{B}_{sing}(\phi)$.
\end{lemma}

Notice that in the following proposition, by extending our set from $\mathcal{B}_{sing}(\phi)$ to 
$\mathcal{B}_{\mathsf{Fix}_+}(\phi)$, 
we can generalize from generic leaves to include nongeneric leaves also.

\begin{prop}\label{asym2}
Let $\phi$ be rotationless and hyperbolic and
 suppose $l',l''$ are two leaves in $\cup \Lambda_i$, where $\Lambda_i$'s are attracting laminations
 of $\phi$. If some end of $l'$ is asymptotic to some end of $l''$ then
 both $l',l'' \in\mathcal{B}_{\mathsf{Fix}_+}(\phi)$.

 To summarize, only lines in $\mathcal{B}_{\mathsf{Fix}_+}(\phi)$ can be asymptotic.
\end{prop}

\begin{proof}

 Let $f:G\longrightarrow G$ be a CT representing $\phi$. Denote the endpoints of $l'$ by $Q', P$ and
 the endpoints of $l''$ by $Q'', P$, where $P$ is the common endpoint.

 If both the leaves are generic then we are done by Lemma \ref{asym1}. If both $l', l''$ are nongeneric then
 they are already in $\mathcal{B}_{\mathsf{Fix}_+}(\phi)$ by Lemma \ref{nongen}. It remains to consider the case when
 $l'$ is nongeneric leaf but $l''$ is generic. The following claim completes the proof.

 \textbf{Claim:} If $l'$ is a nongeneric leaf and $l''$ is a generic leaf which are asymptotic then 
 $l''\in\mathcal{B}_{\mathsf{Fix}_+}(\phi)$.

 \textit{Proof of claim:} Denote the endpoints of $l', l''$ by $Q', P, Q''$ respectively, where $P$
 is the common endpoint.
Let $\Lambda^+_j$  denote the attracting lamination that contains $l''$. Then both $l', l''$ are in
$\mathcal{B}_{na}(\Lambda^-_j)$. Recall that the set of lines in $\mathcal{B}_{na}(\Lambda^-_j)$ are closed
under concatenation ( by Lemma \ref{concat} ). So consider the line $l$ obtained by concatenation of $l'$ and $l''$, which has
endpoints at $Q', Q''$. Then by using the work of Handel and Mosher
\ref{concat} , there exists a principal lift $\Phi$ of $\phi$ such that either all three of $P, Q', Q''$ are in
$\mathsf{Fix}_+(\widehat{\Phi})$ or all three points are in $\partial_{ext}(A,\phi)$
for some $[A]\in\mathcal{A}_{na}(\Lambda^\pm_j)$. 

If the first conclusion is true then $l''\in\mathcal{B}_{sing}(\phi)$ by 
definition. 
If the second
conclusion is true, then $l''$ cannot have an endpoint carried by $\mathcal{A}_{na}(\Lambda^\pm_j)$
because it is generic  leaf of $\Lambda^+_j$ and so the only remaining possibility is that  
$P, Q''$ are both in $\mathsf{Fix}_+(\phi)$ . 
This gives us that $l''\in\mathcal{B}_{\mathsf{Fix}_+}(\phi)$.


 \end{proof}

 We now use this Proposition to prove our main result about Cannon-Thurston maps for hyperbolic $\phi$.

\begin{thm} \label{ct1}
Consider the exact sequence of hyperbolic groups:
$$ 1\rightarrow \mathbb{F} \rightarrow G\rightarrow \langle\phi\rangle\rightarrow 1$$

  The Cannon-Thurston map $\widehat{i}:\partial\mathbb{F}\
 \rightarrow \partial G$ is a finite-to-one map and cardinality of preimage of each point in $\partial G$ is
 bounded above by a number depending only on $\mathsf{rank}(\mathbb{F})$.

\end{thm}

\begin{proof}
The map $\widehat{i}$ identifies two points in $\partial\mathbb{F}$ if and only if there is a line in
$\widetilde{\mathcal{WL}}(\phi) \cup \widetilde{\mathcal{WL}}(\phi^{-1})$ which connects the points. If three or more
points are identified then we get asymptotic lines in either $\widetilde{\mathcal{WL}}(\phi)$ or $\widetilde{\mathcal{WL}}(\phi^{-1})$
(not both).

The above proposition \ref{asym2} tells us that only the lines in $\mathcal{B}_{\mathsf{Fix}_+}(\phi)$ or 
$\mathcal{B}_{\mathsf{Fix}_-}(\phi)$ can be asymptotic.
But for any given $\phi$, both $\mathcal{B}_{\mathsf{Fix}_+}(\phi)$ and $\mathcal{B}_{\mathsf{Fix}_-}(\phi)$ are finite sets with cardinality
uniformly bounded above by Proposition \ref{singfin}.

\end{proof}

\subsection{Quasiconvexity in extension of $\mathbb{F}$ by $\phi$}
\label{sec:A2}
Next we proceed to show another application of our work. It is related to quasiconvexity of subgroups
of $\mathbb{F}$ in the extension group $G$. First we quote a result due to Mitra:

\begin{lemma} \cite[Lemma 2.1]{Mj-99} \cite[Lemma 2.4]{MjR-17}
\label{qc}
 Consider the exact sequence $$ 1\rightarrow \mathbb{F} \rightarrow G\rightarrow \mathcal{H}\rightarrow 1$$
 of hyperbolic groups.
If $H$ is a finitely generated infinite index subgroup of
 $\mathbb{F}$, then $H$ is quasiconvex in $G$ if and only if it does not carry a leaf of the ending lamination.
\end{lemma}

We use our description of the ending lamination of hyperbolic $\phi$ to prove an interesting result about
quasiconvexity of finitely generated subgroups of $\mathbb{F}$.
The first step is to connect the concept of filling in the sense of free factor supports with the
concept of ``minimally filling''. (see section \ref{sec:3})

Kapovich and Lustig in their work of ending laminations \cite[Proposition A.2]{KL-13} showed that for a hyperbolic fully irreducible
outer automorphism, its ending lamination set is \textit{minimally filling}. A set of lines $S$ is said to be minimally filling if
there does not exist any finitely generated infinite index subgroup of $\mathbb{F}$ that carries a line of $S$.
Note that minimally filling implies \textit{filling} in the sense of free factor support \ref{fill}. The converse
is generally not true. However something interesting can be said for the converse direction too.

\begin{prop}\label{fill-minfill}
 Suppose $S$ is a set of lines in $\mathbb{F}$ that is filling in the sense of free factor supports
 and there does not exist any free factor of any finite index subgroup of $F$ that carries a line in $S$
 then $S$ is minimally filling in $F$. Converse also holds.

\end{prop}

\begin{proof}

 Using Corollary 1 from the work  of \cite{Burns-69}, we know that any finitely generated subgroup of $F$ can be
 realized as a free factor of some finite index subgroup of $F$. This together with the additional hypothesis
 in the lemma implies $S$ is minimally filling.

 The converse part follows directly from definitions and the observation that any free factor of any finite index
 subgroup of $F$ is finitely generated and infinite index in $F$.
 \end{proof}

The following definition generalizes the ``minimally filling'' so that we can use it for hyperbolic outer automorphisms
which are not fully irreducible. The idea behind this definition is the property that the free factor support
of lamination is either all of $[\mathbb{F}]$ or every line is carried by some proper free factor system $[F^i]$
where $F^i$ is a proper free factor of $\mathbb{F}$.

\begin{definition} \label{relminfill}
Consider a finitely generated subgroup of infinite index  $H<\mathbb{F}$.

 A set of lines $S$ in $\mathbb{F}$ said to be \textbf{\textit{minimally filling with respect to $H$}} if for every finitely
 generated subgroup $H'<\mathbb{F}$ containing $H$ as a proper free factor, no finitely generated infinite index
 subgroup in $H'$ carries a line of $S$.

 Similarly we say a set of lines $\Lambda$ in $\mathcal{B}$ is
 \textbf{\textit{minimally filling with respect to $H$}} if every lift of $\Lambda$ in $\widetilde{\mathcal{B}}$
 is minimally filling with respect to $H$.

 \end{definition}

 It is easy to see that  $S$ is minimally filling with respect to $H$ if and only if no finitely generated
 subgroup of $\mathbb{F}$ containing $H$ as a proper free factor , carries a line in $S$. However
 we will use the aforementioned definition since it is easier to relate with the standard definition of
 minimally filling. The following lemma establishes the connection.

 \begin{lemma}

  If $S$ is minimally filling in $\mathbb{F}$ if and only if $S$
  is minimally filling with respect to every finitely generated, infinite index subgroup of $\mathbb{F}$.

  \end{lemma}

  \begin{proof}
  To see the forward direction, let $\mathbb{F}\geq H'\geq H$
  be finitely generated subgroups such that $H$ is a proper free factor of $H'$. If $H'$ carries a line in $S$
  then there exists a proper free factor $K<H'$ of $H'$ which carries this line. But then we know that $K$ must
  have infinite index in $H'$, hence in $\mathbb{F}$, which violates that $S$ is minimally filling.

  Conversely, suppose $S$
  is minimally filling with respect to every finitely generated, infinite index subgroup of $F$. If there exists
  some finitely generated infinite index subgroup $H<\mathbb{F}$ which carries a line in $S$, then by
  using the result of \cite{Burns-69}, we know that there exists a finite index (hence finitely generated)
  subgroup $H'<\mathbb{F}$ which contains $H$ as a free factor (if $H$  is a free factor of $\mathbb{F}$ then
  take $H' = \mathbb{F}$).
  But $H$ has infinite index in $H'$ since
  it is a proper free factor of $H'$ and $H'$ has finite index in $\mathbb{F}$,
  hence $S$ is not minimally filling with respect to $H$. Which is a contradiction.
 \end{proof}

Let us elaborate on the purpose  of giving this definition.


\begin{prop}\label{minfil}

 Consider the exact sequence $$ 1\rightarrow \mathbb{F} \rightarrow G\rightarrow \langle\phi\rangle\rightarrow 1$$
 where $\phi\in\out$ hyperbolic.

Let  $H$ be a finitely generated, infinite index subgroup of $\mathbb{F}$. If every attracting and every
repelling lamination of $\phi$ is minimally filling with respect to $H$ then $H$ is quasiconvex in $G$.

\end{prop}

\begin{proof}
 Suppose $H$ is not quasiconvex. Then by Lemma \ref{qc} it carries a leaf of the ending lamination set.
Without loss we may assume that it carries a line in $\widetilde{\mathcal{WL}}(\phi)$. Since the closure of projection of any line in
$\widetilde{\mathcal{WL}}(\phi)$ contains an attracting lamination and $\partial H$ is compact,
$H$ must carry lift of a generic leaf of some $\Lambda_i\in\mathcal{L}(\phi)$.

Using \cite{Burns-69} one has a finite index subgroup $H'<\mathbb{F}$ such that $H$ is a proper free factor
of $H'$ (since $H$ is infinite index in $\mathbb{F}$), and hence infinite index in $H'$ and $H$ carries lift
of a leaf of $\Lambda_i$.
This violates that $\Lambda_i$ is minimally filling with respect to $H$. Hence $H$ must be
quasiconvex in $G$.

\end{proof}
The following result establishes certain dynamical conditions for
a finitely generated, infinite index $H<\mathbb{F}$ to be quasiconvex in the extension group $G$.
We give one equivalent condition and one sufficient condition for $H$ to be quasiconvex.

\begin{thm} \label{qc1}
 Consider the exact sequence $$ 1\rightarrow \mathbb{F} \rightarrow G\rightarrow \langle\phi\rangle\rightarrow 1$$
 where $\phi$ is a hyperbolic automorphism of $\mathbb{F}$. Let $H$ be a finitely generated infinite index subgroup of
 $\mathbb{F}$.  Then the following are equivalent:
 \begin{enumerate}
  \item $H$ is quasiconvex in $G$
  \item $\partial H\cap \{\mathsf{Fix}_+(\phi)\cup \mathsf{Fix}_-(\phi)\} = \emptyset$.

 \end{enumerate}
 Moreover both $(1)$ and $(2)$ are satisfied if every attracting and repelling lamination of $\phi$ is
 minimally filling with respect to $H$.
\end{thm}

\begin{proof}
The conclusion in the ``moreover'' part is discussed in Proposition \ref{minfil}. We proceed to prove the equivalence of
$(1)$ and $(2)$.

Suppose $H$ is quasiconvex in $G$.
Since $H$ is finite rank subgroup of $\mathbb{F}$, $\partial H$ is compact in $\partial \mathbb{F}$
the set of lines
carried by $[H]$ is closed in the weak topology. If
$P\in\partial H\cap \{\mathsf{Fix}_+(\phi)\cup \mathsf{Fix}_-(\phi)\} \neq \emptyset$ then weak closure
of $P$ must contain either an attracting lamination or a repelling lamination for $\phi$. Hence
$H$ contains leaves of the ending lamination set and this contradicts that $H$ quasiconvex by using Lemma \ref{qc}

Conversely, if $\partial H\cap \{\mathsf{Fix}_+(\phi)\cup \mathsf{Fix}_-(\phi)\} = \emptyset$,
then $H$ does not carry a leaf of any ending lamination set, hence $H$ is quasiconvex.
To see this use Theorem \ref{Limit} and Lemma \ref{structure} which states that every attracting (repelling)
lamination carries a singular line which has endpoints in $\mathsf{Fix}_+(\phi)$ ($\mathsf{Fix}_-(\phi)$) .

\end{proof}

We make the following observation regarding the special case when $\phi$ is fully irreducible. This
case is well known and was done in \cite{Mj-99}, \cite{KL-13}.

\begin{cor}\label{qciwip}
 If $\phi$ is fully irreducible, then any finitely generated infinite index subgroup is quasiconvex in the
 extension group $G$.
\end{cor}

\begin{proof}
  This follows from the work of  \cite[Proposition 2.4]{BFH-97} and
   \cite[Proposition A.2]{KL-13} where both papers prove that
 $\Lambda$ is minimally filling in $\mathbb{F}$ and hence minimally filling with respect to any
 finitely generated infinite index subgroup of $\mathbb{F}$, where $\Lambda$ is the unique attracting lamination
 of $\phi$. Similarly for the unique repelling lamination.
\end{proof}

Now we proceed to give an algebraic condition for quasiconvexity:

\begin{thm}\label{qc2}
Consider the exact sequence $$ 1\rightarrow \mathbb{F} \rightarrow G\rightarrow \langle\phi\rangle\rightarrow 1$$
 where $\phi$ is a hyperbolic automorphism of $\mathbb{F}$.

 If $H$ is a finitely generated, infinite index subgroup of $\mathbb{F}$ such that:
 \begin{enumerate}
  \item $H$ is contained in a proper free factor of $\mathbb{F}$.
 \item $H$ does not contain any subgroup which is also a subgroup of some
 $\phi$ periodic free factor.

 \end{enumerate}

 Then $H$ is quasiconvex in $G$.
 Moreover if $\phi$ does not have an attracting lamination that fills in the sense of free factor support, condition
 $(2)$ is a sufficient condition for quasiconvexity of $H$ in $G$.
\end{thm}

\begin{proof}
 Suppose $H$ is not quasiconvex. Then by Lemma \ref{qc}, $H$ carries a leaf of the ending lamination set.
 Since the attracting and repelling laminations are paired by free factor supports, (see \ref{sec:6})
 without loss assume that $H$ carries lift of a line in  $\mathcal{WL}(\phi)$. Since $H$ is of finite rank, $\partial H$ is compact and
 set of lines carried by $[H]$ is closed in the weak topology. This implies that $H$ carries lift of an attracting
 lamination $\Lambda_i\in\mathcal{L}(\phi)$. Let $[F^i]$ be the
 (necessarily $\phi$-invariant, see proof of Lemma 3.2.4, \cite{BFH-00}) free factor support of $\Lambda_i$.

 \textbf{Case 1: $F^i$ is not proper:} In this case we use condition $(1)$ to get a contradiction. Since
 $H$ is contained in a proper free factor $H'$ say, the free factor support of lines carried by $H$ is contained
 in $H'$. This implies $F^i < H''$ for some conjugate of $H'$, but since $F^i = \mathbb{F}$ we get a contradiction.

 \textbf{Case 2: $F^i$ is a proper free factor:} In this case we conclude that $H\cap gF^ig^{-1}$ is nonempty for
 some $g\in\mathbb{F}$ and hence a
 subgroup of a $\phi$ invariant free factor. This violates condition $(2)$.

 The ``moreover'' part is a consequence of the proof of Case 2.

\end{proof}

\subsection{Ending Laminations for purely atoroidal $\mathcal{H}$}
\label{sec:A3}
  We now proceed to define the notion of weak limit set of a group $\mathcal{H}\in\out$, which is
  purely atoroidal i.e. every element of $\mathcal{H}$ is a hyperbolic outer automorphism. For this we will use the
  ideas developed so far to prove the special case in \ref{Limit} and use the notion of ending laminations as a guide.
  It is to be noted that extending the definition from $\widetilde{\mathcal{WL}}(\phi)$ to $\widetilde{\mathcal{WL}}(\mathcal{H})$
  is not entirely obvious at a first glance, however familiarity with description of the ending lamination set (see remark \ref{flip})
  will be of great help.
  So for the rest of the section let us fix the following notations and assumptions:

  \textbf{Assumptions: }
  \begin{enumerate}
  \item $1\rightarrow \mathbb{F}\rightarrow G\rightarrow \mathcal{H}\rightarrow 1 $ is an exact sequence
  of hyperbolic groups where $\mathcal{H}$ is non-elementary and purely atoroidal.
  \item $z\in\partial\mathcal{H}$ is point in the Gromov boundary of $\mathcal{H}$.
  \item $\phi_n\rightarrow z$ means $\phi_n$ is a sequence of vertices in the Cayley graph of $\mathcal{H}$ so that
   they lie on a geodesic joining $1$ to $z$ in the compactified space
  $\mathcal{H}\cup\partial\mathcal{H}$. So whenever we use the notation $\phi_n$  we mean a term in the
  sequence as described here.
  \item If $S$ is a set of lines in $\mathcal{B}$, we denote its lift to $\widetilde{\mathcal{B}}$ by
  $\widetilde{S}$.
  \end{enumerate}

  \begin{definition}\label{limH}
   Let $z\in\partial\mathcal{H}$ and $\phi_n$ be a sequence in $\mathcal{H}$ converging to $z$.
   Define the sets $$\mathcal{WL}(\phi_n, [c]) = \{l\in\mathcal{B} \arrowvert {l \text{ is a weak limit of } \phi^{k}_n}_\#([c])\text{ as } k\rightarrow\infty\}$$

$$ \widetilde{\mathcal{WL}}(z, [c])   = \bigcap_{k = 1}^{\infty} \big\{\overline{\bigcup_{j=k}^{\infty} \widetilde{\mathcal{WL}}(\phi_j, [c])}\big\} $$


\end{definition}

 So a line $\tilde{l}$ is in $\widetilde{\mathcal{WL}}(z, [c])$ if and only if given any finite subpath $\alpha$ of $\tilde{l}$ and $M\geq 1$
 there exists $j>M$ such that $\alpha$ is a subpath of ${\phi^k_j}_\#([c])$ for some $k>0$. This
 is equivalent to saying that every line in $\widetilde{\mathcal{WL}}(z, [c])$ is a weak limit of lines
 $\tilde{l}_j\in\widetilde{\mathcal{WL}}(\phi_j, [c])$.

 This implies that
 $\alpha$  is a subpath of $\phi_j([c'])$ for some $c'\in \mathbb{F}-\{1\}$. Hence we deduce that for any subpath
 $\alpha$ of $\tilde{l}$ and any $M>0$, there exists $j>M$ and a  conjugacy class $[c']$ such that
 $\alpha$ is subpath of ${\phi_j}_\#([c'])$. This implies $\widetilde{\mathcal{WL}}(z, [c]) \subset \overline{\Lambda}_z = \Lambda_z$.
 Hence we have the following corollary:

 \begin{cor}\label{mahan}
 If for $z\in\partial\mathcal{H}$, $\Lambda_z$ is  the ending lamination set  defined by Mitra, then:
  $$\widetilde{\mathcal{WL}}(z,[c])  \subset \Lambda_z$$.
   \end{cor}
Note that $\widetilde{\mathcal{WL}}(z, [c])$  is a closed set by construction.

 As we can see that the ending lamination set is far more complicated when $\mathcal{H}$ is not a cyclic group.
 One of the big problems one faces is if a set of lines that occurs as a attracting or repelling lamination
 for every element of $\mathcal{H}$. Recall that stabilizers of attracting (repelling) laminations
 of a fully irreducible $\phi\in\out$ are virtually cyclic \cite[Theorem 2.14]{BFH-97}. But this may not be true
 if $\phi$ is hyperbolic but reducible even if the lamination fills (in the sense of free factor supports).

 Infact something worse can happen. Since $\mathcal{H}$ is hyperbolic the stabilizer subgroup of lamination
 inside $\mathcal{H}$ if not virtually cyclic will necessarily contain a nonabelian free group. This
 is due to the lemma that Bestvina, Feighn and Handel showed
 Tits alternative holds for $\out$ \cite{BFH-00} and Feighn, Handel proved every abelian
 subgroup of $\out$ is virtually $\mathbb{Z}^n$ for some $n$ \cite{FH-09}. In order to make a meaningful
 conclusion about Cannon-Thurston maps for $G$ we need to avoid these situations. The following results
 guarantee that such things do not occur.\\


 \begin{lemma}\label{wlz}
 Suppose  $\phi, \psi$ are two hyperbolic elements of $\out$ contained in $\mathcal{H}$.
 Denote the ending laminations $\Lambda_{z}, \Lambda_{z'}$, where $z = \phi^{\infty}, z' = \psi^{\infty}$ are
 in $\partial\mathcal{H}$. Let $l, l' \in \Lambda_{z}\cup \Lambda_{z'} $. Also suppose that $\phi$ and $\psi$ do not have 
 a common attracting fixed point. Under these assumptions,
 if $l$, $l' $ have an asymptotic end, then
 \begin{enumerate}
  \item $\phi$ and $\psi$ have a common attracting lamination.
  \item Both $l, l'$ are
  lifts of lines in either $\mathcal{B}_{\mathsf{Fix}_+}(\phi)$ or $\mathcal{B}_{\mathsf{Fix}_+}(\psi)$.
 \end{enumerate}

 \end{lemma}

 \begin{proof}
 Due to the conclusion from Proposition \ref{asym2} we may assume that $l\in \Lambda_{z}, l'\in \Lambda_{z'}$.
 We split the proof into cases:

 \textbf{Case 1: }
 If $l, l'$ are both generic leaves of $\phi$ and $\phi$ and share an endpoint $P$, then the weak closure
 of $P$, $\overline{P} = \Lambda^+_\phi = \Lambda^+_\psi$, where $\Lambda^+_\phi$ and $ \Lambda^+_\psi$ are attracting laminations
 for $\phi$ and $\psi$ respectively. So $l, l'$ can be treated as lifts of generic leaves of $\Lambda^+_\phi$ and
(since both are birecurrent leaves) the conclusion (2) now follows directly from Proposition \ref{asym2}.

 \textbf{Case 2: }
Suppose $l$ is lift of some generic leaf  of $\phi$,  $l'$ is lift of a line
 $\mathcal{B}_{\mathsf{Fix}_+}({\psi})$ and they are asymptotic with common endpoint $P$. 
 Thus $\overline{P} = \Lambda^+_\phi \supseteq \Lambda^+_\psi$. If the inclusion is proper then the leaves of 
 $\Lambda^+_\psi$ are nongeneric leaves of $\Lambda^+_\psi$. But by using Corollary \ref{nongen} we know that 
 these nongeneric leaves are lines in $\mathcal{B}_{\mathsf{Fix}_+}(\phi)$, which is a finite set by using 
 Proposition \ref{singfin}.
 But since $\Lambda^+_\psi$ is an uncountable set we must have $\Lambda^+_\phi=\Lambda^+_\psi$.
 Now apply Proposition \ref{asym2} (on $\psi$) to get the second conclusion.

 \end{proof}

 The proposition below gives some necessary conditions that a pair of elements of $\mathcal{H}$ must satisfy
 in order for the extension group to be hyperbolic. Recall that Dowdall-Taylor proved that
 for the extension group to be hyperbolic, every element of $\mathcal{H}$ must be hyperbolic.

 \begin{prop} \label{nec}
  If $\phi, \psi \in \mathcal{H}$ are two points such that $\phi$, $\psi$ do not have a common power
  , then the following equivalent conditions are satisfied:
  \begin{enumerate}
 \item  $\phi^\infty , \psi^\infty$ are distinct points in
  $\partial\mathcal{H}$
 \item They do not have a common attracting lamination.
  \item They do not have common attracting fixed point.
  \item No leaf of an attracting lamination or singular line of $\phi$ is asymptotic to a leaf of an attracting lamination
 or singular line for $\psi$.

 \end{enumerate}

   \end{prop}

   \begin{proof}
    Using \cite[Proposition 5.1]{Mj-97} we know that $\phi^\infty$ and $\psi^\infty$ are distinct points
    if and only if no line in $\Lambda_{\phi^\infty}$ has any endpoint common with a line in $\Lambda_{\psi^\infty}$.
    Hence they do not have a common attracting lamination by using Proposition \ref{el1}. The reverse
    direction is true by Theorem \ref{Limit}.
    This establishes the equivalence between (1) and (2).
    The equivalence between (1) and the others follow similarly by using \ref{Limit} and \ref{el1}.

   \end{proof}

\begin{remark}
Note that the above proposition combined with Mitra's result \cite[Proposition 5.1]{Mj-97} and
Theorem \ref{Limit} shows that
for any $z\in\partial\mathcal{H}$ which is not the endpoint of an axis, its corresponding ending lamination set $\Lambda_z$ does not contain
any line which is equal or asymptotic to a leaf or a singular line of any element of $\mathcal{H}$.
This shows that the techniques we used for showing finiteness of fibers of Cannon-Thurston maps
cannot be directly applied in the more general setting when $\mathcal{H}$ is not (virtually) cyclic.
\end{remark}

We end this paper with a question: Is the converse to Corollary \ref{mahan} true ?

\noindent\rule[0.5ex]{\linewidth}{1pt}
Address: Department of Mathematical Sciences, IISER Mohali, Punjab, India.\\
Contact: \href{mailto:pritam@scarletmail.rutgers.edu}{pritam@scarletmail.rutgers.edu}
\bibliographystyle{plainnat}
\def\bibfont{\small}
\bibliography{biblo}

\begin{thebibliography}{20}
\providecommand{\natexlab}[1]{#1}
\providecommand{\url}[1]{\texttt{#1}}
\expandafter\ifx\csname urlstyle\endcsname\relax
  \providecommand{\doi}[1]{doi: #1}\else
  \providecommand{\doi}{doi: \begingroup \urlstyle{rm}\Url}\fi

\bibitem[Bestvina and Feighn(1992)]{BF-92}
M.~Bestvina and M.~Feighn.
\newblock {A combination theorem for negetively curved groups}.
\newblock \emph{J. Differential Geom.}, 43\penalty0 (4):\penalty0 85--101,
  1992.

\bibitem[Bestvina and Handel(1992)]{BH-92}
M.~Bestvina and M.~Handel.
\newblock {Train tracks and {A}utomorphisms of {F}ree groups}.
\newblock \emph{Ann. of Math.}, 135:\penalty0 1--51, 1992.

\bibitem[Bestvina et~al.(1997)Bestvina, Handel, and Feighn]{BFH-97}
M.~Bestvina, M.~Handel, and M.~Feighn.
\newblock {Laminations, {T}rees, and irreducible automorphisms of free groups}.
\newblock \emph{Geom. Func. Anal.}, 7:\penalty0 215--244, 1997.

\bibitem[Bestvina et~al.(2000)Bestvina, Feighn, and Handel]{BFH-00}
M.~Bestvina, M.~Feighn, and M.~Handel.
\newblock {Tits alternative in Out($F_n$)-I: {D}ynamics of
  {E}xponentially-growing {A}utomorphisms}.
\newblock \emph{Ann. of Math.}, 151:\penalty0 517--623, 2000.

\bibitem[Brinkmann(2000)]{Br-00}
P.~Brinkmann.
\newblock {Hyperbolic Automorphisms of Free Groups}.
\newblock \emph{Geom. Funct. Anal.}, 10\penalty0 (5):\penalty0 1071--1089,
  2000.

\bibitem[Burns(1999)]{Burns-69}
R.~Burns.
\newblock {}.
\newblock \emph{Proc. Amer. Math. Soc.}, 127:\penalty0 1625--1631, 1999.

\bibitem[Dowdall et~al.(2016)Dowdall, Kapovich, and Taylor]{DKT-16}
S.~Dowdall, I.~Kapovich, and S.~Taylor.
\newblock {Cannon-{T}hurston maps for hyperbolic {F}ree group extensions}.
\newblock \emph{Israel journal of Math.}, 216:\penalty0 743--797, 2016.

\bibitem[Feighn and Handel(2000)]{FH-09}
M.~Feighn and M.~Handel.
\newblock {Abelian subgroups of {O}ut({F}n))}.
\newblock \emph{Geometry and Topology}, 13:\penalty0 1657--1727, 2000.

\bibitem[Feighn and Handel(2011)]{FH-11}
M.~Feighn and M.~Handel.
\newblock {The recognition theorem for Out($F_n)$}.
\newblock \emph{Groups Geom. Dyn}, 5:\penalty0 39--106, 2011.

\bibitem[Gaboriau et~al.(1998)Gaboriau, Levitt, Jaeger, and Lustig]{GJLL-98}
D.~Gaboriau, G.~Levitt, A.~Jaeger, and M.~Lustig.
\newblock {An index for counting fixed points of automorphisms of free groups}.
\newblock \emph{Duke Math. J}, 93\penalty0 (3):\penalty0 425--452, 1998.

\bibitem[Handel and Mosher(2011)]{HM-11}
M.~Handel and L.~Mosher.
\newblock {Axes in Outer space.}
\newblock \emph{Mem. Amer. Math. Soc}, 213\penalty0 (1004), 2011.

\bibitem[Handel and Mosher(2017{\natexlab{a}})]{HM-13a}
M.~Handel and L.~Mosher.
\newblock {Subgroup decomposition in Out($F_n$), Part I: Geometric Models}.
\newblock \emph{Mem. Amer. Math. Soc}, (accepted), 2017{\natexlab{a}}.
\newblock URL \url{http://arxiv.org/abs/1302.2378}.

\bibitem[Handel and Mosher(2017{\natexlab{b}})]{HM-13c}
M.~Handel and L.~Mosher.
\newblock {Subgroup decomposition in Out($F_n$), Part III: Weak Attraction
  theory}.
\newblock \emph{Mem. Amer. Math. Soc.}, (accepted), 2017{\natexlab{b}}.
\newblock URL \url{http://arxiv.org/find/math/1/au:+Mosher_L/0/1/0/all/0/1}.

\bibitem[Handel and Mosher(2009)]{HM-09}
M.~Handel and Lee Mosher.
\newblock {Subgroup classification in Out($F_n$)}.
\newblock \emph{arXiv}, 2009.
\newblock URL \url{http://arxiv.org/abs/0908.1255}.

\bibitem[Kapovich and Lustig(2015)]{KL-13}
I.~Kapovich and M.~Lustig.
\newblock {Cannon-Thuston maps for irreducible automorphisms of $F_n$}.
\newblock \emph{J. London. Math. Soc}, 91:\penalty0 203--224, 2015.

\bibitem[Levitt and Lustig(2008)]{LL-08}
G.~Levitt and M.~Lustig.
\newblock {Automorphisms of free groups have asymptotically periodic dynamics}.
\newblock \emph{J. Reine Angew. Math.}, 619:\penalty0 1--36, 2008.

\bibitem[Mitra(1997)]{Mj-97}
M.~Mitra.
\newblock {Ending laminations for hyperbolic group extensions}.
\newblock \emph{Geom. Func. Anal.}, 7\penalty0 (2):\penalty0 379--402, 1997.

\bibitem[Mitra(1998)]{Mj-98}
M.~Mitra.
\newblock {Cannon-Thurston maps for hyperbolic group extensions}.
\newblock \emph{Topology}, 37\penalty0 (3):\penalty0 527--538, 1998.

\bibitem[Mitra(1999)]{Mj-99}
M.~Mitra.
\newblock {On a theorem of Scott and Swarup}.
\newblock \emph{Proc. Amer. Math. Soc., (arXiv:1209.4165 corrected version)},
  127:\penalty0 1625--1631, 1999.

\bibitem[Mj and Rafi(2017)]{MjR-17}
M.~Mj and K.~Rafi.
\newblock {Algebraic Ending Laminations and Quasiconvexity}.
\newblock \emph{arXiv:1506.08036}, 2017.

\end{thebibliography}

 \end{document}